\documentclass[a4paper]{article}
\usepackage{comment}
\usepackage{graphicx}
\newif\iffigure
\ifx11 
\figuretrue
\else
\figurefalse
\fi

\usepackage{algorithm}
\usepackage[noend]{algpseudocode}
\usepackage{amsmath}
\usepackage{amssymb}
\usepackage{amsthm}
\usepackage{braket}
\usepackage{cite}
\usepackage{comment}
\usepackage{enumerate}
\usepackage{here}
\usepackage{hyperref}
\usepackage{latexsym}
\usepackage{mathtools}

\newtheorem{theorem}{Theorem}[section]
\newtheorem{method}[theorem]{Method}
\newtheorem{lemma}[theorem]{Lemma}
\newtheorem{remark}[theorem]{Remark}
\newtheorem{corollary}[theorem]{Corollary}

\newenvironment{acases}{\left\{\begin{aligned}}{\end{aligned}\right.}

\def\Real{\mbox{$\mathbb{R}$}}

\def\0{\boldsymbol{0}}
\def\1{\boldsymbol{1}}
\def\2{\boldsymbol{2}}
\def\3{\boldsymbol{3}}
\def\4{\boldsymbol{4}}
\def\5{\boldsymbol{5}}
\def\6{\boldsymbol{6}}
\def\7{\boldsymbol{7}}
\def\8{\boldsymbol{8}}
\def\9{\boldsymbol{9}}
\def\a{\boldsymbol{a}}
\def\b{\boldsymbol{b}}

\def\l{\boldsymbol{l}}

\def\q{\boldsymbol{q}}

\def\t{\boldsymbol{t}}
\def\u{\boldsymbol{u}}
\def\v{\boldsymbol{v}}
\def\w{\boldsymbol{w}}
\def\x{\boldsymbol{x}}

\def\A{\boldsymbol{A}}

\def\D{\boldsymbol{D}}
\def\E{\boldsymbol{E}}

\def\I{\boldsymbol{I}}
\def\J{\boldsymbol{J}}

\def\L{\boldsymbol{L}}
\def\M{\boldsymbol{M}}
\def\N{\boldsymbol{N}}
\def\O{\boldsymbol{O}}

\def\U{\boldsymbol{U}}

\def\X{\boldsymbol{X}}

\newcommand{\futoji}[1]{\boldsymbol{#1}}
\def\flambda{\futoji{\lambda}}
\def\fGamma{\futoji{\Gamma}}
\def\fOmega{\futoji{\Omega}}
\def\fomega{\futoji{\omega}}

\begin{document}

\title{Sparsity Exploitation of Accelerated Modulus-Based Gauss-Seidel Method for Interactive Rigid Body Simulations}

\author{
	Shugo Miyamoto
		\thanks{Department of Systems Innovation, School of Engineering, The University of Tokyo. Email:miyamoto-s@g.ecc.u-tokyo.ac.jp.}, and
	Makoto Yamashita
		\thanks{Department of  Mathematical and Computing Science, Tokyo Institute of Technology. Email: Makoto.Yamashita@c.titech.ac.jp.}
}
\date{October 22, 2019}

\maketitle

\begin{abstract}
Large-scale linear complementarity problems (LCPs) are repeatedly solved in interactive rigid-body simulations.
The projected Gauss-Seidel method is often employed for LCPs, since it has advantages in computation time, numerical robustness, and memory use.
Zheng and Yin (2013) proposed modulus-based matrix splitting iteration methods and showed their effectiveness for large problems, but a simple application of their approach to large-scale LCPs in interactive rigid-body simulations is not effective since such a simple application demands large matrix multiplications.

In this paper, we propose a novel method derived from accelerated modulus-based matrix splitting iteration methods that fits LCPs arising in interactive rigid-body simulations.
To improve the computation time, we exploit sparsity structures related to the generalized velocity vector of rigid bodies.
We discuss the convergence of the proposed method for an important case that the coefficient matrix of a given LCP is a positive definite matrix.
Numerical experiments show that the proposed method is more efficient than the simple application of the accelerated modulus-based Gauss-Seidel method and that the accuracy in each step of the proposed method is superior to that of the projected Gauss-Seidel method.

\end{abstract}

{\bf Keywords:} Iterative methods for linear systems, Dynamics of multibody systems

\section{Introduction}
\label{intro}
In rigid-body simulations, interactions (for example, normal forces) between rigid bodies are often mathematically modeled as constraints.
For computing these constraint forces, we usually need to solve certain equations.
Two representative constraint formulations for rigid-body simulations are
acceleration-based formulations \cite{baraff1993non} and velocity-based formulations \cite{anitescu1997formulating}.
In the acceleration-based formulations, the constraints are described with forces and accelerations of rigid bodies; we first compute forces and accelerations, then integrate them to obtain velocity changes.
On the other hand, in the velocity-based formulations, the variables in the constraints are impulses and velocities of the rigid bodies.
In this paper,
we focus on the velocity-based formulations, since the velocity-based formulations are widely used and are known to be superior to the acceleration-based formulations in many aspects (for example, see \cite{erleben2004stable, poulsen2010heuristic}).

There are two main categories for solving constraints,
iterative approaches and direct approaches.
We focus on the iterative approaches rather than the direct approaches, since the direct approaches such as pivoting methods often suffer from time complexity and numerical instability as pointed in~\cite{nakaoka2007constraint}.
In the impulse-based iterative approaches \cite{erleben2004stable}, impulses are applied to the rigid bodies sequentially, until certain convergence conditions are satisfied.

Contact constraints are frequently modeled in a form of complementarity problems \cite{baraff1994fast}; in particular,  linear complementarity problems (LCPs) give mathematical formulations for frictionless contacts.
Among various iterative methods for solving LCPs, the projected Gauss-Seidel (PGS) method \cite{bender2014interactive} has a remarkable flexibility, therefore the PGS method or its extensions are often employed to solve contact constraints \cite{erleben2004stable, poulsen2010heuristic, nakaoka2007constraint}.
Another iterative approach for solving LCPs is the use of modulus-based methods.
Bai~\cite{bai2010modulus} established modulus-based matrix splitting iteration (MMSI) methods, which include modulus-based Jacobi (MJ), modulus-based Gauss-Seidel (MGS), modulus-based successive over relaxation (MSOR), and modulus-based accelerated overrelaxation (MAOR) iteration methods as special cases.
Recently, Zheng and Yin \cite{zheng2013accelerated} proposed an accelerated modulus-based matrix splitting iteration (AMMSI) method as an improvement of Bai~\cite{bai2010modulus}.
In a similar way to the MMSI methods, the AMMSI methods include accelerated modulus-based Jacobi (AMJ), accelerated modulus-based Gauss-Seidel (AMGS), accelerated modulus-based SOR (AMSOR), and accelerated modulus-based accelerated overrelaxation (AMAOR) iteration methods.
To our best knowledge, no application of the MMSI or AMMSI methods to real-time simulations has been examined before.

Since the AMGS method is not designed for interactive rigid-body simulations, a simple application of the AMGS method causes inefficiency in computation and it is a serious disadvantage for real-time simulations.
In many applications of real-time simulations, interactive computer graphics  and operations are considered most important. If the computation in each simulation step is considerably slower than real time, the quality of the users' experience would be seriously degraded.

In this paper, we resolve this difficulty by focusing the update formula in the AMGS method and exploiting the structures related to the generalized velocity vector of rigid bodies.

We also give a theoretical proof on the convergence of the AMGS method. Bai~\cite{bai2010modulus} already discussed the convergence,
but the assumption in \cite{bai2010modulus} is too restrictive
to apply the same discussion to rigid-body simulations.
We extend the proof of \cite{bai2010modulus} to cover
an important case that the coefficient matrix in the LCP is a positive definite matrix.

Through numerical experiments, we observed that the proposed AMGS method that exploited the sparsity attained shorter computation time than the original AMGS method. Furthermore its convergence rate in each iteration was better than that of the PGS method. These results indicate that the proposed method is useful for practical real-time simulations.

The outline of this paper is as follows.
In Section~\ref{sec:pre}, we briefly introduce a formulation of velocity-based constraints as an LCP. We also discuss the projected Gauss-Seidel method and the AMGS method to solve LCPs.
The application of the AMGS method to rigid-body simulations is developed in Section~\ref{sec:amgs4rbs}, and we prove convergence theorems of the AMGS method in Section~\ref{sec:convthm}.
In Section~\ref{sec:numexp}, we will show numerical results to verify the efficiency of the AMGS method.
Finally, we will give a conclusion in Section~\ref{sec:conclusion}.

\section{Preliminaries}\label{sec:pre}

\subsection{Linear complementarity problem with velocity-based constraints}
For the latter discussions, we briefly introduce an LCP that arises from velocity-based constraints. 
For more details, the readers can refer to \cite{baraff1994fast, stewart1996implicit, tonge2012mass}.

During a rigid-body simulation, we keep tracking movements of the rigid bodies
in multiple time periods, therefore,
an entire simulation is divided into a sequence of simulation steps, and each simulation step corresponds to a small time step.
Since the constraints on the rigid bodies should be satisfied at each time,
we solve the following LCP  in each simulation step:
\begin{align}
\left\{\begin{array}{rcl}
\flambda & \ge & \0 \\
\J\M^{-1}\J^T \flambda +\J \v + \b &\geq& \0 \\
(\J\M^{-1}\J^T \flambda + \J \v + \b)^T \flambda &=& 0
\end{array}\right. \label{eq:LCP}
\end{align}
In this paper, we use the superscript $T$ to denote the transpose of a vector or a matrix.

The decision variable in this LCP is
$\flambda\in\mathbb{R}^m$ which is the impulse vector applied to the rigid bodies in the constraint space. 
The first constraint in \eqref{eq:LCP} requires $\flambda$ to be nonnegative to ensure that the constraint impulse must be repulsive.

The second constraint in \eqref{eq:LCP} corresponds to the velocity constraints in the rigid-body simulation.
The vectors $\b\in\mathbb{R}^m$ and $\v \in \mathbb{R}^n$
are the bias vector in a constraint space and
the generalized velocity vector of rigid bodies, respectively.
More precisely,
when we have $N$ rigid bodies, $\v$ is a vector that consists of $N$ linear and angular velocities, \textit{i.e.},
\begin{align*}
\v=\begin{pmatrix}
\v_1^T & \fomega_1^T & \v_2^T & \fomega_2^T & \cdots & \v_N^T & \fomega_N^T
\end{pmatrix}^T
\end{align*}
where $\v_i\in\mathbb R^3$ and $\fomega_i\in\mathbb R^3$ are the linear velocity and the angular velocity of the $i$th rigid body, respectively, for $i=1,\ldots,N$;
thus the length of $\v$ is $n = 6 N$.
The matrix $\J \in \Real^{m \times n}$ is the Jacobian matrix corresponding to the velocity constraints.
The generalized mass matrix of the rigid bodies
$\M\in\mathbb R^{n\times n}$ consists of masses and inertia tensor matrices in the diagonal positions:
\begin{align}
\M=\begin{pmatrix}
m_1\E_3 &     &      &     &         &      &     \\
& \I_1 &      &     &         &      &     \\
&     & m_2\E_3 &     &         &      &     \\
&     &      & \I_2 &         &      &     \\
&     &      &     & \ddots  &      &     \\
&     &      &     &         & m_N\E_3 &     \\
&     &      &     &         &      & \I_N \\
\end{pmatrix}, \label{eq:structure-M}
\end{align}
where $m_i\in\mathbb{R}$ and $\I_i\in\mathbb R^{3\times3}$ are the mass and the inertia tensor matrix of the $i$th rigid body, respectively.
We also use $\E_r \in\mathbb R^{r\times r}$ to denote the identity matrix of order $r$.
The inertia tensor matricies $\I_1, \ldots, \I_m$ are symmetric, so is $\M$.

The third constraint in \eqref{eq:LCP} is a complementarity condition.
We can understand  this complementarity condition as follows.
If $(\J\M^{-1}\J^T\flambda + \J\v+\b)_i>0$ holds for some $i$, then
the rigid bodies are moving away from each other in the direction of the $i$th constraint, therefore,
the $i$th constraint should be ``inactive''.
However, $\flambda$ is the impulse vector, thus $\lambda_i>0$ implies that the $i$th constraint must be ``active''.
Hence, $\lambda_i>0$ and $(\J\M^{-1}\J^T\flambda + \J\v+\b)_i>0$ should not hold simultaneously,
and this requirement is implemented in the complementarity condition.

By denoting $\q = \J\v+\b$  and $\A = \J\M^{-1}\J^T$ and introducing an auxiliary variable $\w \in \Real^m$, the LCP \eqref{eq:LCP} can be expressed in  a general LCP as follows:
\begin{align}
\label{eq:lcp1}
LCP(\q, \A)
\quad
\left\{\begin{array}{rcl}
\A\flambda+\q& = & \w\\
\w^T\flambda&=& 0 \\
\flambda, \w &\geq & \0
\end{array}\right.
\end{align}
It is known that if the constraints are non-degenerate, the Jacobian matrix $\J$ is full row rank and the matrix $\A=\J\M^{-1}\J^T$ is positive definite
(see \cite{bender2014interactive}, for example).
Throughout this paper, we assume that $\A$ is positive definite.

At the end of this subsection, we should note that
the input data $\b, \v, \J$ and $\M$ vary in accordance with simulation steps.
If we express the time dependence explicitly, they should be
$\b^{(t)}, \v^{(t)}, \J^{(t)}$ and $\M^{(t)}$ where $t$ is the simulation step.
However, in this paper, we mainly focus on solving \eqref{eq:LCP} in each simulation step,
therefore, we usually drop the simulation step $(t)$ from $\b^{(t)}, \v^{(t)}, \J^{(t)}$ and $\M^{(t)}$.

\subsection{Projected Gauss-Seidel method}

In the LCP \eqref{eq:lcp1} from the rigid-body simulation,
the matrix $\A$ has a structure such that $\A = \J\M^{-1}\J^T$.
The projected Gauss-Seidel (PGS) method~\cite{erleben2004stable} is
designed to solve more general LCPs \eqref{eq:lcp1} in the sense that
the assumption for $\A$ is only positive definiteness.
The PGS method is an iterative method and generates a sequence
$\left\{\flambda\right\}_{k=0}^{\infty} \subset \Real^m$.

A key property in the PGS method is to decompose $\A$ into $\A=\D-\L-\U$ such that $\D$ is a diagonal matrix, $\L$ a strictly lower triangular matrix, and $\U$ a strictly upper triangular matrix. Since we assume $\A$ is a positive definite matrix, $\D$ is invertible. Due to this decomposition,
$\A \flambda + \q = \0$ is equivalent to
$\flambda = \D^{-1} (\L \flambda + \U \flambda - \q).$

Taking this formula and the complementarity condition into consideration, the PGS method computes the next iteration $\flambda^{k+1}$ by the following update formula:
\begin{align}
\label{eq:pgs-base}
\flambda^{k+1}=\max\left\{\0, \D^{-1}\left(\L\flambda^{k+1}+\U\flambda^k-\q\right)\right\}.
\end{align}
Throughout this paper, we use $\max\left\{\a, \b\right\}$ ($\min\left\{\a, \b\right\}$) to denote
the element-wise maximum (minimum, respectively) of two vectors $\a$ and $\b$.
The PGS method continues the update by \eqref{eq:pgs-base}
until the sequence $\left\{\flambda^k\right\}_{k=0}^\infty$ converges enough, or the number of the iterations reaches a certain limit.

In the rigid-body simulation, the initial vector $\flambda^0$ is usually set as a zero vector $\0$ or the impulse vector obtained in the previous simulation step.
Since the initial vector often affects the performance of iterative approaches, the use of the solution from the previous simulation step makes the convergence faster~\cite{erleben2004stable}. Such a technique is called warm start.

\subsection{Accelerated modulus-based Gauss-Seidel method}
\label{sec:amgs}

For solving the general LCP \eqref{eq:lcp1},
Bai \cite{bai2010modulus} devised the following implicit fixed-point equation that is essential for the modulus-based matrix splitting iteration (MMSI) methods:
\begin{align}
\label{eq:fixed-point1}
(\M_0\fGamma+\fOmega_1)\x=(\N_0\fGamma-\fOmega_2)\x+(\fOmega-\A\fGamma)|\x|-\q
\end{align}
Here, the matrices $\M_0 \in \Real^{m \times m}$ and $\N_0 \in \Real^{m \times m}$
are a splitting pair of $\A$ such that $\A = \M_0 - \N_0$.
$\fOmega \in \Real^{m \times m}$ and $\fGamma \in \Real^{m \times m}$ are
two diagonal matrices whose diagonal entries are positive.
$\fOmega_1 \in \Real^{m\times m}$ and $\fOmega_2 \in \Real^{m\times m}$ are nonnegative diagonal matrices such that $\fOmega=\fOmega_1+\fOmega_2$.
We should emphasize that the variable in \eqref{eq:fixed-point1} is $\x$, and
we use $|\x|$ to denote the element-wise absolute values of $\x$.
The relation
between $\x$ and the pair of $\flambda$ and $\w$ in \eqref{eq:lcp1} will be discussed in
Theorem~\ref{thm:fixedpoint}.

By setting $\fOmega_1=\fOmega$, $\fOmega_2=\O$, and $\fGamma=\frac1\gamma \E_m$ with a parameter $\gamma>0$ in \eqref{eq:fixed-point1},
we can obtain a simplified implicit fixed-point equation:
\begin{align}
\label{eq:fixed-point1-2}
(\M_0+\gamma \fOmega)\x=\N_0 \x+(\gamma \fOmega-\A)|\x|-\gamma \q.
\end{align}
Based on \eqref{eq:fixed-point1-2}, the iteration of the MMSI method can be derived as follows:
\begin{align}
(\M_0+\gamma \fOmega)\x^{k+1}=\N_0 \x^k+(\gamma \fOmega-\A)|\x^k|-\gamma \q.
\label{eq:MMSIupdate}
\end{align}

We decompose $\A$ into $\A=\D-\L-\U$ in the same way as the PGS method.
Set $\gamma=2$, and let $\alpha >0$ and $\beta > 0$ be two parameters.
Then, we can derive the update formula of four methods from \eqref{eq:MMSIupdate}; the modulus-based Jacobi (MJ) by setting $\M_0=\D$ in \eqref{eq:MMSIupdate}, the modulus-based Gauss-Seidel (MGS) by $\M_0=\D-\L$, the modulus-based successive over relaxation (MSOR) by $\M_0=\frac1\alpha \D-\L$, and the modulus-based accelerated overrelaxation (MAOR) iteration method by $\M_0=\frac1\alpha(\D-\beta \L)$, respectively.

Zheng and Yin~\cite{zheng2013accelerated}
utilized two splitting pairs of the matrix $\A$ such that $\A=\M_1-\N_1=\M_2-\N_2$,
and
devised a new equation based on \eqref{eq:fixed-point1}:
\begin{align}
\label{eq:fixed-point2}
(\M_1\fGamma+\fOmega_1)\x=(\N_1\fGamma-\fOmega_2)\x+(\fOmega-\M_2\fGamma)|\x|+\N_2\fGamma|\x|-\q.
\end{align}

Zheng and Yin~\cite{zheng2013accelerated}  established the following theorem to show an equivalence between \eqref{eq:fixed-point2} and
$LCP(\q,\A)$ in \eqref{eq:lcp1}.
Since a detailed proof is not given in \cite{zheng2013accelerated}, we give the proof here.
\begin{theorem}\upshape\cite{zheng2013accelerated}
\label{thm:fixedpoint}
The following statements hold between \eqref{eq:fixed-point2} and
$LCP(\q,\A)$:
\begin{enumerate}[(i)]
  \item if $(\flambda,\w)$ is a solution of $LCP(\q,\A)$, then $\x=\frac12(\fGamma^{-1}\flambda-\fOmega^{-1}\w)$ satisfies
  \eqref{eq:fixed-point2}.
  \item if $\x$ satisfies \eqref{eq:fixed-point2}, then
  the pair of $\flambda=\fGamma(|\x|+\x)$ and $\w=\fOmega(|\x|-\x)$ is a solution of $LCP(\q,\A)$.
\end{enumerate}
\end{theorem}
\begin{proof}
We first prove (i).
Since $(\flambda,\w)$ is a solution of $LCP(\q,\A)$,
$(\flambda,\w)$ satisfies the four constraints,
$\A\flambda + \q = \w$,
$\w^T \flambda = 0$, $\flambda \ge \0$ and $\w \ge \0$.
The first constraint $\A\flambda + \q = \w$ is equivalent to
\begin{align*}
 (\fOmega+\A\fGamma)(\fGamma^{-1}\flambda-\fOmega^{-1}\w) &= (\fOmega-\A\fGamma)(\fGamma^{-1}\flambda+\fOmega^{-1}\w) - 2\q.
\end{align*}
From the rest three constraints and the fact that $\fGamma$ and
$\fOmega$ are diagonal matrices whose diagonal entries are positive,
if $\x=\frac12(\fGamma^{-1}\flambda-\fOmega^{-1}\w)$, it holds that
$|\x| =\frac12(\fGamma^{-1}\flambda+\fOmega^{-1}\w)$.
Therefore, $\x$ satisfies
\begin{align}
(\fOmega+\A\fGamma)\x &= (\fOmega-\A\fGamma)|\x| - \q \label{eq:fixed-point2-proof1}
\end{align}
and this is equivalent to \eqref{eq:fixed-point2}.

To prove (ii), from \eqref{eq:fixed-point2-proof1}, it holds that
$\A\fGamma(|\x|+\x)+\q = \fOmega(|\x|-\x)$.
By the relations $\flambda=\fGamma(|\x|+\x)$ and $\w=\fOmega(|\x|-\x)$, we obtain $\A\flambda+\q=\w$.
Since $\fGamma$ and $\fOmega$ are positive diagonal matrices,
it is easy to check that $\flambda$ and $\w$ are nonnegative vectors.
Finally, it is also easy to show
the element-wise complementarity between $\flambda$ and $\w$.
\hfill \qed
\end{proof}

We may use Theorem~\ref{thm:fixedpoint} to establish some iterative methods for solving $LCP(\q,\A)$, but we need to set appropriate matrices for the implicit fixed-point equation \eqref{eq:fixed-point2}
 in actual computations.
In particular, the splitting pair of $\fOmega$ is not unique.
By fixing $\fOmega_1= \fOmega$, $\fOmega_2=\O$ and $\fGamma=\frac{1}{\gamma} \E_m$, we
derive a simplified update equation of \eqref{eq:fixed-point2} as follows:
\begin{align}
\label{eq:simplified-fixed-point2}
(\M_1+\gamma \fOmega)\x=\N_1\x+(\gamma \fOmega-\M_2)|\x|+\N_2|\x|-\gamma \q.
\end{align}
Based on this equation, Zheng and Yin~\cite{zheng2013accelerated}
provided an update formula of the AMMSI methods:
\begin{align}
\label{eq:generic-iteration}
(\M_1+\gamma \fOmega)\x^{k+1}=\N_1\x^k+(\gamma \fOmega-\M_2)|\x^k|+\N_2|\x^{k+1}|-\gamma \q
\end{align}
When the sequence
$\left\{\x^k\right\}_{k=0}^\infty$ converges enough,
the AMMSI methods output the impulse vector by using the relation
$\flambda = \fGamma (|\x| + \x) = \frac{|\x| + \x}{\gamma}$.

By changing the splitting pairs of $\A$,
the update formula~\eqref{eq:generic-iteration} above yields variant methods;
MMSIM ($\M_2=\A$ and $\N_2=\O$),
the accelerated modulus-based Jacobi (AMJ) iteration method
($\M_1=\D$, $\N_1=\L+\U$, $\M_2=\D-\U$ and $\N_2=\L$),
the accelerated modulus-based SOR (AMSOR) iteration method
($\M_1=\frac1\alpha \D-\L$, $\N_1=\left(\frac1\alpha-1\right)\D+\U$, $\M_2=\D-\U$ and $\N_2=\L$),
and the accelerated modulus-based accelerated overrelaxation (AMAOR) iteration method ($\M_1=\frac1\alpha(\D-\beta \L)$, $\N_1=\frac1\alpha((\alpha-1)\D+(\alpha-\beta)\L+\alpha \U)$, $\M_2=\D-\U$ and $\N_2=\L$).

In particular, the update formula of the accelerated modulus-based Gauss-Seidel (AMGS) method in~\cite{zheng2013accelerated}
is derived with
$\M_1=\D-\L$, $\N_1=\U$, $\M_2=\D-\U$ and $\N_2=\L$ as follows:
\begin{align}
(\D + \gamma \fOmega-\L)\x^{k+1}=\U\x^k+(\gamma \fOmega-\D+\U)|\x^k|+\L|\x^{k+1}|-\gamma \q. \label{eq:AMGS1}
\end{align}
Let $\Delta \x^k = \x^{k+1} - \x^k$ be
the difference between $\x^k$ and $\x^{k+1}$.
Then, \eqref{eq:AMGS1}
is equivalent to
\begin{align}
\begin{array}{rcl}
(\D+\gamma \fOmega)\Delta \x^k
&=& \L\x^{k+1}-(\gamma \fOmega+\D-\U)\x^k+(\gamma \fOmega-\D+\U)|\x^k| \\
& & +\L|\x^{k+1}|- \gamma \q.
\end{array}
\label{eq:amgs-iter2}
\end{align}
By Theorem~\ref{thm:fixedpoint} and $\fGamma = \frac{1}{\gamma} \E_m$, the sequence $\{\flambda^k\}_{k=0}^\infty$ for the LCP~\eqref{eq:lcp1} can be associated with the sequence $\{\x^k\}_{k=0}^\infty$ generated by \eqref{eq:AMGS1} by the relation
$\flambda^k=\frac{|\x^k|+\x^k}{\gamma}=\frac{2}{\gamma}\max\{\0,\x^k\}$,
thus $\flambda^k$ is a multiple of  the positive part of $\x^k$.
This motivates us to split $\x^k$ into the positive and negative parts such that  $\x^k=\x_+^k - \x_-^k$, where
$\x_+^k = \max\{\0,\x^k\} = \frac12(|\x^k| + \x^k)$ and
$\x_-^k=-\min\{\0,\x^k\} = \frac12(|\x^k|-\x^k)$.
From the relations $\x^k = \frac{\gamma}{2} \flambda^k - \x_-^k$ and
$|\x^k| = \frac{\gamma}{2} \flambda^k  + \x_-^k$,
\eqref{eq:amgs-iter2} is equivalent to
\begin{align*}
\begin{array}{rcl}
(\D+\gamma \fOmega)\Delta \x^k
&=&\gamma \L\flambda^{k+1}-(\gamma \fOmega+\D-\U)(\frac{\gamma}{2}\flambda^k-\x_-^k)\\
& & +(\gamma\fOmega-\D+\U)(\frac{\gamma}{2}\flambda^k+\x_-^k)-\gamma \q \\
&=&\gamma \L\flambda^{k+1}- \gamma (\D - \U) \flambda^k + 2 \gamma  \fOmega \x_-^k - \gamma \q.
\end{array}
\end{align*}
Therefore, for computing $\Delta x_i^k$, we only need the $i$th component of $\x_-^k$, which will be denoted as $(\x_-^k)_i$,
since $\D$ and $\fOmega$ are diagonal matrices.
This simplifies the computation of $\Delta \x^k$.
Recalling the decomposition of $\A = \D - \L - \U$,
we compute $\Delta x^k_i$ for each $i=1,\ldots,m$ by
\begin{align}
\Delta x^k_i&=\frac{\gamma}{D_{ii}+\gamma \Omega_{ii}}\left(\sum_{j=1}^{i-1}L_{ij}\lambda^{k+1}_j-D_{ii}\lambda^k_i+\sum_{j=i+1}^mU_{ij}\lambda^k_j+2\Omega_{ii}(\x_-^k)_i-q_i\right) \nonumber \\
&=-\frac{\gamma}{A_{ii}+\gamma \Omega_{ii}}\left(\sum_{j=1}^{i-1}A_{ij}\lambda^{k+1}_j+\sum_{j=i}^mA_{ij}\lambda^k_j+q_i-2\Omega_{ii}(\x_-^k)_i\right).
\label{eq:delta-x-update}
\end{align}

We can summarize a framework of the AMGS method as follows.
\begin{method}\upshape
	\label{mtd:amgs}
	\textbf{(the AMGS method for \eqref{eq:lcp1})}

	Choose a nonnegative vector  $\flambda^0\in\mathbb R^m$
	as an initial vector.
	Generate the iteration sequence $\{\flambda^k\}_{k=0}^\infty$ by the following procedure:
	\begin{algorithmic}
		\State $\x^0 \gets \frac{\gamma}{2} \flambda^0$
		\State $k\gets0$
		\Repeat \ \ {$k=1,2,\ldots$}
		\For {$i=1,2,\ldots,m$}
		\State $x_i^{k+1} \gets x_i^k-\frac{\gamma}{A_{ii}+\gamma \Omega_{ii}}\left(\sum_{j=1}^{i-1}A_{ij}\lambda^{k+1}_j+\sum_{j=i}^mA_{ij}\lambda^k_j+q_i-2\Omega_{ii}(\x_-^k)_i\right)$
		\EndFor
		\State $\flambda^k = \frac{2}{\gamma} \x_+^k$.
		\Until {$\x^k$ satisfies a certain convergence threshold}
	\end{algorithmic}
\end{method}

\section{Accelerated modulus-based matrix splitting iteration methods for interactive rigid-body simulation}
\label{sec:amgs4rbs}
In this section, we propose a numerical method to solve LCPs that arises from interactive rigid-body simulations using the AMGS method.
Since the AMGS method is a generic method for LCPs, it is possible to simply apply the AMGS method to \eqref{eq:lcp1}.
However, explicit evaluation of $\A=\J\M^{-1}\J^T$ is inefficient even though the matrices $\J$ and $\M^{-1}$ are sparse. Thus, such a simple application of the AMGS method is not practical. As we will discuss in later, this is mainly because the number of the non-zero elements in $\A$ becomes very large.
To overcome this inefficiency, we modify the AMGS method so that it does not require the explicit evaluation of the matrix $\A$.
In a similar way to the AMGS method, we can also modify the AMSOR method for solving LCPs in the rigid-body simulations.

As already pointed out at above,  the direct computation of $\A$ is unfavorable for real-time simulations.
In the viewpoint of the computation cost, we should avoid the computations
of
$\sum_{j=1}^{i-1}A_{ij}\lambda^{k+1}_j$ and $\sum_{j=i}^mA_{ij}\lambda^k_j$, which involve all the off-diagonal elements of $\A$.

To improve the computation efficiency,
we introduce an intermediate variables $\flambda^{k+1,i} \in \Real^m$ and $\v^{k+1,i} \in \Real^{n}$. In particular, $\v^{k+1,i}$ stores information of applied impulse \cite{erleben2004stable, tonge2012mass}.
For $i=1,\ldots,m$,
let
\begin{align}
\label{eq:pgs-lambda-intermediate}
\flambda^{k+1,i}&=\begin{pmatrix}
\lambda_1^{k+1} & \ldots & \lambda_{i-1}^{k+1} & \lambda_i^k & \ldots & \lambda_m^k
\end{pmatrix}^T\\
\label{eq:pgs-v-intermediate}
\v^{k+1,i}&=\v+\widehat{\M} \flambda^{k+1,i}
\end{align}
where $\widehat{\M} = \M^{-1}\J^T.$
By the definitions of $\A$ and $\q$, it holds
\begin{align*}
&\sum_{j=1}^{i-1}A_{ij}\lambda_j^{k+1}+\sum_{j=i}^mA_{ij}\lambda_j^k+q_i
= \left(\A \flambda^{k+1,i}\right)_i + q_i \\
=&\left(\J\M^{-1}\J^T\flambda^{k+1,i}\right)_i+(\J\v)_i+b_i
=\left(\J\v^{k+1,i}\right)_i+b_i =\sum_{j=1}^m J_{ij}v_j^{k+1,i}+b_i,
\end{align*}
and this leads to
\begin{align*}
x_i^{k+1}=x_i^k+\Delta x^k_i=x_i^k-\frac{\gamma}{A_{ii}+\gamma \Omega_{ii}}\left(\sum_{j=1}^m J_{ij}v_j^{k+1,i}+b_i-2\Omega_{ii}(\x_-^k)_i\right).
\end{align*}

Due to the relation $\flambda^{k} =
\frac{2}{\gamma} \x_+^k$,
we can compute $\flambda^{k+1, i+1}$
by updating only the $i$th position of $\flambda^{k+1, i}$,
\begin{align*}
\begin{array}{rcl}
\flambda^{k+1, i+1}
&=&\begin{pmatrix}
\lambda_1^{k+1} & \ldots & \lambda_{i-1}^{k+1} & \frac{2}{\gamma}(\x_+^{k+1})_i &
\lambda_{i+1}^{k} & \ldots & \lambda_m^k
\end{pmatrix}^T.\\
\end{array}
\end{align*}
Thus, from \eqref{eq:pgs-v-intermediate}, we obtain
\begin{align*}
\begin{array}{rcl}
\v^{k+1, i+1} &=& \v^{k+1,i}
+ \widehat{\M} (\flambda^{k+1, i+1}  - \flambda^{k+1, i}) \\
&=& \v^{k+1,i} + \widehat{\M}_{*i} (\lambda_i^{k+1, i+1} - \lambda_i^{k+1}),
\end{array}
\end{align*}
where $\widehat{\M}_{*i}$ is the $i$th column of $\widehat{\M}$.

The following method summarizes the proposed AMGS method for rigid-body simulations.
\begin{method}\upshape
\label{mtd:proposedamgs}
\textbf{(the proposed AMGS method for rigid-body simulations)}

Choose a nonnegative vector  $\flambda^0\in\mathbb R^m$
as an initial vector.
Generate the iteration sequence $\{\flambda^k\}_{k=0}^\infty$ by the following procedure:
\begin{algorithmic}
\State $\x^0 \gets \frac{\gamma}{2} \flambda^0$
\State $\flambda^1 \gets \flambda^0$
\State $\widehat{\M} = \M^{-1} \J^T$.
\State $\v^{1,1}\gets \v+\widehat{\M} \flambda^0$
\State $k\gets0$
\Repeat \ \ {$k=1,2,\ldots$}
	\For {$i=1,2,\ldots,m$}
		\State $x_i^{k+1} \gets x_i^k-\frac{\gamma}{A_{ii}+\gamma \Omega_{ii}}\left(\sum_{j=1}^m J_{ij}v_j^{k+1,i}+b_i-2\Omega_{ii}(\x_-^k)_i\right)$
		\State $\lambda_i^{k+1}\gets \frac{2}{\gamma} \max\left\{0, x_i^{k+1}\right\}$
		\If {$i=m$}
			\State $\v^{k+2,1}\gets \v^{k+1,i}+\widehat{\M}_{*i}(\lambda_i^{k+1}-\lambda_i^k)$
		\Else
			\State $\v^{k+1,i+1}\gets \v^{k+1,i}+\widehat{\M}_{*i}(\lambda_i^{k+1}-\lambda_i^k)$
		\EndIf
	\EndFor
\Until {$\x^k$ satisfies a certain convergence threshold}
\end{algorithmic}
\end{method}

We compare the computation costs of Method~\ref{mtd:amgs} and Method~\ref{mtd:proposedamgs}.
In Method~\ref{mtd:amgs}, we update $\x^{k+1} = \x^k + \Delta \x^k$,
and for each $i=1,\ldots, m$, we compute
$\sum_{j=1}^{i-1}A_{ij}\lambda^{k+1}_j$ and $\sum_{j=i}^mA_{ij}\lambda^k_j$ in \eqref{eq:delta-x-update}.
Therefore, if $\A \in \Real^{m \times m}$ is fully dense,
the computation cost to obtain $\x^{k+1}$ is $\mathcal{O}(m^2)$.

In Method~\ref{mtd:proposedamgs}, we can exploit the structures of $\J \in \Real^{m \times n}$ and $\M \in \Real^{n \times n}$.
We use $nnz(\J)$ to denote the number of nonzero elements in $\J$.
Since the matrix $\M$ is composed with $\E_3$ and $\I_1, \ldots, \I_N \in \Real^{3 \times 3}$ at the diagonal positions as shown in \eqref{eq:structure-M}, we know that $nnz(\widehat{\M}) = nnz(\M^{-1} \J^T) \le 3 nnz(\J)$.
Actually, each row of $\widehat{\M}$ is a multiple of one column of $\J$
or a linear combinations of three columns of $\J$.
To obtain $x_i^{k+1}$ from $x_i^k$,
the computation of $\sum_{j=1}^m J_{ij} v_j^{k+1}$ is required for $i=1,\ldots,m$,
thus the computation cost in the $k$th outer iteration amounts to $\mathcal{O}(nnz(\J))$.
Similarly,
to obtain $\v_{k+1}^{i+1}$ from $\v_{k+1,i}$, we need
$\widehat{\M}_{*i}(\lambda_i^{k+1}-\lambda_i^k)$,
thus the computation cost of this part for each $k$ is
$\mathcal{O}(nnz(\widehat{\M}))$, and this is same as
$\mathcal{O}(nnz(\J))$.
Consequently,
the computation cost to obtain $\x^{k+1}$ from $\x^k$
in Method~\ref{mtd:proposedamgs} is $\mathcal{O}(nnz(\J))$.
In the rigid-body simulation, $nnz(\J)$ is $\mathcal{O}(m)$ and is smaller than $\mathcal{O}(m^2)$,
thus we can expect Method~\ref{mtd:proposedamgs} is much faster
than the direct use of $\A$ in Method~\ref{mtd:amgs}. We will verify this efficiency
in the numerical experiments of Section~\ref{sec:numexp}.

We should mention that the matrix $\A$ is not always fully-dense, and we can actually find some zero elements in $\A$ of the test instances that will be used in Section~\ref{sec:numexp}. However, the positions of nonzero elements cannot be determined before the multiplication $\A = \J \M^{-1} \J^T$. In addition, even if we skip the zero elements of $\A$ in Method~\ref{mtd:amgs}, the effect is less significant than Method~\ref{mtd:proposedamgs} and the numerical efficiency of Method~\ref{mtd:amgs} is still insufficient for real-time simulations.

\subsection{Linear Complementarity Problems with Lower and Upper Bounds}
\label{sec:boxedlcp}

In this section, we discuss a method for solving LCPs with lower and upper bounds on $\flambda$, which often occur in the formulations of contact constraints with friction \cite{tonge2012mass, poulsen2010heuristic}. We call such LCPs with lower and upper bounds ``Boxed LCPs (BLCPs)''.
Consider the following BLCP:
\begin{align*}
BLCP(\q,\l,\u,\A)\begin{acases}
\A\flambda+\q &=\w \\
\l \leq \flambda &\leq \u\\
\text{for each}\ i=1,\ldots,m,&
\begin{cases}
w_i\geq0&\text{if \ }\lambda_i=l_i\\
w_i\leq0&\text{if \ }\lambda_i=u_i\\
w_i=0&\text{if \ }l_i < \lambda_i<u_i.
\end{cases}
\end{acases}
\end{align*}
Here, $\l \in \Real^m$ and $\u \in \Real^m$ are the lower and the upper bounds, respectively. Without loss of generality, we assume $0 \le l_i < u_i$ for each $i=1,\ldots,m$.

We define a projection function of $\flambda$  to the interval
$\l$ and $\u$ by
\begin{align*}
p_{BLCP}(\flambda) = \min\left\{\max\left\{\l, \flambda \right\},\u \right\}.
\end{align*}
Since $\l$ is a nonnegative vector, so is $p_{BLCP}(\flambda)$.
With this projection, we can give a AMGS method for Boxed LCPs as follows.
\begin{method}\upshape
\label{mtd:amgs-boxed}
	\textbf{(an AMGS method for Boxed LCPS in rigid-body simulations)}

	Choose a nonnegative vector  $\flambda^0\in\mathbb R^m$
	as an initial vector.
	Generate the iteration sequence $\{\flambda^k\}_{k=0}^\infty$ by the following procedure:
	\begin{algorithmic}
		\State $\flambda^0 \gets p_{BLCP}(\flambda^0$)
		\State $\x^0 \gets \frac{\gamma}{2} \flambda^0$
		\State $\flambda^1 \gets \flambda^0$
		\State $\widehat{\M} = \M^{-1} \J^T$
		\State $\v^{1,1}\gets \v+\widehat{\M} \flambda^0$
		\State $k\gets0$
		\Repeat \ \ {$k=1,2,\ldots$}
		\For {$i=1,2,\ldots,m$}
		\State $x_i^{k+1} \gets x_i^k-\frac{\gamma}{A_{ii}+\gamma \Omega_{ii}}\left(\sum_{j=1}^m J_{ij}v_j^{k+1,i}+b_i-2\Omega_{ii}(\x_-^k)_i\right)$
		\State $\lambda_i^{k+1}\gets \frac{2}{\gamma} p_{BLCP}(x_i^{k+1})$
		\If {$i=m$}
		\State $\v^{k+2,1}\gets \v^{k+1,i}+\widehat{\M}_{*i}(\lambda_i^{k+1}-\lambda_i^k)$
		\Else
		\State $\v^{k+1,i+1}\gets \v^{k+1,i}+\widehat{\M}_{*i}(\lambda_i^{k+1}-\lambda_i^k)$
		\EndIf
		\EndFor
		\Until {$\x^k$ satisfies a certain convergence threshold}
	\end{algorithmic}
\end{method}

Note that $BLCP(\q,\l,\u,\A)$ with $\l = \0$ and $\u$ being a very large vector represents the same problem as $LCP(\q,\A)$, so Method~\ref{mtd:amgs-boxed} can be considered as a generalization of Method~\ref{mtd:amgs}.

\section{Convergence theorem}
\label{sec:convthm}
In this section, we focus on the convergence of the accelerated modulus-based Gauss-Seidel (AMGS) method.
Zheng and  and Yin~\cite{zheng2013accelerated} discussed only the two cases:
(i) $\A$ is a positive definite matrix and $\fOmega$ is a multiple of the identity matrix ($\fOmega  =\bar{\omega} \E_m$ for some $\bar{\omega} > 0$), and (ii) $\A$ is an $H_+$-matrix.

Since the matrix $\A$ is always positive definite in the rigid-body simulation, we
extend the proof in \cite{zheng2013accelerated}
so that we can handle a more general form of $\fOmega$.
For example,
if we can take $\fOmega = \alpha \D$, we may be able to improve the convergence, as mentioned later in Remarks~\ref{rem:bestalpha} and \ref{rem:bestomega}.
Here, $\alpha\in \Real$ is a positive constant,
and $\D$ is the diagonal matrix whose diagonal elements are
those of $\A$, thus $\fOmega = \alpha \D$ is not covered by \cite{zheng2013accelerated}.

In order to establish the convergence theorem of the AMGS method with $\fOmega$ which is not always a multiple of the identity matrix, we need Lemma~\ref{lem:lower} below.
We use $\|\x \| = \sqrt{\x^T \x}$ to denote the Euclidean norm of $ \x \in \mathbb R^m$.
We also use $\|\X\|$ to denote the spectral norm of a matrix $\X\in\mathbb R^{m\times m}$,
thus we can employ $\| \X \x \| \le \| \X \| \| \x \|$.
In addition
let $diag(\X)$  denote the diagonal matrix whose diagonal elements correspond to those of a matrix $\X\in\mathbb R^{m\times m}$, \textit{i.e.},
\begin{align*}
diag(\X)=\begin{pmatrix}
X_{11} &        &        &        \\
& X_{22} &        &        \\
&        & \ddots &        \\
&        &        & X_{mm}
\end{pmatrix}.
\end{align*}

\begin{lemma}\upshape
    \label{lem:lower}
    Let $\L_1, \L_2\in\mathbb R^{m \times m}$ be lower triangular matrices, and $\widehat{\L}\in\mathbb R^{m\times m}$ be a non-singular lower triangular matrix.
    Then, the following equations hold:
    \begin{align}
    \label{eq:lem-lower-1}
    \|\L_1\|&=\|diag(\L_1)\|\\
    \label{eq:lem-lower-2}
    diag(\L_1\L_2)&=diag(\L_1)diag(\L_2)\\
    \label{eq:lem-lower-3}
    diag(\widehat{\L}^{-1})&=(diag(\widehat{\L}))^{-1}
    \end{align}
\end{lemma}
\begin{proof}
    We first prove \eqref{eq:lem-lower-1}.
    As $\L_1$ is a lower triangular matrix,
    the eigenvalues of $\L_1$ are $(\L_1)_{11}, (\L_1)_{22}, \ldots, (\L_1)_{mm}$,
    and these values are also those of the diagonal matrix  $diag(\L_1)$.
     Since the spectral norm of a matrix is the maximum absolute value of its eigenvalues,
    we have $\|\L_1\|=\|diag(\L_1)\|$.

    Next, we consider \eqref{eq:lem-lower-2}.
    Focusing the $i$th diagonal element of $\L_1\L_2$, we have
    \begin{align*}
    (\L_1\L_2)_{ii}&=\sum_{j=1}^m (\L_1)_{ij} (\L_2)_{ji}
    &=\sum_{j\in \{1,2,\ldots,i\} \cap \{i,i+1,\ldots,m\}} (\L_1)_{ij} (\L_2)_{ji}
    &=(\L_1)_{ii} (\L_2)_{ii},
    \end{align*}
    therefore, it holds that $diag(\L_1\L_2)=diag(\L_1)diag(\L_2)$.

    Finally, we prove \eqref{eq:lem-lower-3}.
    By \eqref{eq:lem-lower-2} and the fact that the inverse matrix of a lower triangular matrix is also a lower triangular matrix, we obtain
    \begin{align*}
    diag(\widehat{\L})diag(\widehat{\L}^{-1})=diag(\widehat{\L}\widehat{\L}^{-1})=\E_m.
    \end{align*}
    Therefore, $diag(\widehat{\L}^{-1})$ is the inverse matrix of $diag(\widehat{\L})$.
\end{proof}

We are now ready to establish the convergence theorem. The theorem covers the case that $\fOmega=\alpha \D$, which will be actually used in numerical experiments of Section~\ref{sec:numexp}.
For the subsequent discussion, we define
$\tau = \|(\D - \L)^{-1} \U \|$.

\begin{theorem}\upshape
\label{thm:amgs-conv}
Let
\begin{align*}
\delta =  2  \| (\D + \gamma \fOmega)^{-1} \D \| \tau  + \|(\D + \gamma \fOmega)^{-1}(\gamma \fOmega - \D) \|.
\end{align*}
If $\delta <1$, the iteration sequence $\{\flambda ^k\}_{k=0}^\infty\subset\mathbb R^m$ generated by Method~\ref{mtd:proposedamgs} with
an arbitrary nonnegative initial vector $\flambda^0\in\mathbb R^m$
converges to the unique solution $\flambda^*\in\mathbb R^m$ of
$LCP(\q, \A)$.
\end{theorem}
\begin{proof}

    In the AMMSI methods \cite{zheng2013accelerated} from which the AMGS was derived,
    we chose $\fOmega_1 = \fOmega, \fOmega_2 = \O$ and $\fGamma = \frac{1}{\gamma} \E_m$
    for the simplified fixed-point equation \eqref{eq:simplified-fixed-point2}.
    Let $(\flambda^*,\w^*)\in\mathbb{R}^m \times \mathbb{R}^m$ be a solution of $LCP(\q,\A)$,
    and let $\x^*=\frac12 (\gamma \flambda^* - \fOmega^{-1} \w^*)$.
    From Theorem~\ref{thm:fixedpoint},
    $\x^*$
    satisfies \eqref{eq:simplified-fixed-point2},  and
    the convergence of $\left\{\flambda^k \right\}_{k=0}^\infty$ to
    $\flambda^*$ can be guaranteed by
    that of $\left\{\x^k \right\}_{k=0}^\infty$ to
    $\x^*$. Therefore, we focus the convergence of
    the sequence $\left\{\x^k \right\}_{k=0}^\infty$.
    Subtracting \eqref{eq:simplified-fixed-point2} with $\x^*$ from the update formula \eqref{eq:generic-iteration},
    we obtain
    \begin{equation*}
    \label{eq:err-rel-old}
    \begin{aligned}
    (\M_1+\gamma \fOmega)(\x^{k+1}-\x^*)={}&\N_1(\x^k-\x^*)+(\gamma \fOmega-\M_2)(|\x^k|-|\x^*|)\\
    &+\N_2(|\x^{k+1}|-|\x^*|).
    \end{aligned}
    \end{equation*}
    In the setting of the AMGS method ($\M_1 = \D - \L, \N_1 = \U, \M_2 = \D - \U$ and $\N_2 = \L$), we can further evaluate this inequality as follows:
    \begin{equation}
    \begin{aligned}
    (\x^{k+1}-\x^*)={}&(\D - \L + \gamma \fOmega)^{-1} \U(\x^k-\x^*)\\
    & +(\D - \L + \gamma \fOmega)^{-1}(\gamma \fOmega-\D + \U)(|\x^k|-|\x^*|)\\
    &+(\D - \L + \gamma \fOmega)^{-1} \L(|\x^{k+1}|-|\x^*|).
    \end{aligned}
    \end{equation}
    Using the triangular inequality $\| |\a| - |\b| \| \le \|\a - \b\| $ for any two vectors $\a$ and $\b$ of the same dimension, we have
\begin{equation}
\begin{aligned}
\| \x^{k+1}-\x^* \| \le {}&\|(\D - \L + \gamma \fOmega)^{-1} \U \| \|\x^k-\x^* \|\\
& +\| (\D - \L + \gamma \fOmega)^{-1}(\gamma \fOmega-\D + \U)\| \|\x^k-\x^*\|\\
&+\| (\D - \L + \gamma \fOmega)^{-1} \L \| \|\x^{k+1}-\x^*\| \\
\le {} &\|(\D - \L + \gamma \fOmega)^{-1} \U \| \|\x^k-\x^* \|\\
& +\| (\D - \L + \gamma \fOmega)^{-1}(\gamma \fOmega-\D + \L - \L + \U)\| \|\x^k-\x^*\|\\
&+\| (\D - \L + \gamma \fOmega)^{-1} \L \| \|\x^{k+1}-\x^*\|\\
\le {} &2 \|(\D - \L + \gamma \fOmega)^{-1} \U \| \|\x^k-\x^* \|\\
& +\| (\D - \L + \gamma \fOmega)^{-1}(\gamma \fOmega-\D+\L)\| \|\x^k-\x^*\|\\
&+ \| (\D - \L + \gamma \fOmega)^{-1} \L \| (\|\x^{k}-\x^*\|
+ \|\x^{k+1}-\x^*\|).
\end{aligned}\label{eq:err-rel}
\end{equation}
Since $\D$ and $\fOmega$ are diagonal matrices
and $\L$ is a strictly lower triangular matrix,
$\D-\L + \gamma \fOmega$ is a lower triangular matrix,
so is $(\D-\L + \gamma \fOmega)^{-1}$.
Furthermore, the product of two lower triangular matrices
is a lower triangular matrix, therefore
$(\D - \L + \gamma \fOmega)^{-1} (\D - \L)$
is a lower triangular matrix.
Using Lemma~\ref{lem:lower}, it holds that
\begin{align}
\|(\D - \L + \gamma \fOmega)^{-1} (\D - \L)\| \nonumber
=& \| diag((\D - \L + \gamma \fOmega)^{-1} (\D - \L)) \| \nonumber  \\
=& \| diag(\left(\D - \L + \gamma \fOmega\right)^{-1}) diag(\D - \L) \| \nonumber  \\
=& \| \left(diag(\D - \L + \gamma \fOmega)\right)^{-1} diag(\D - \L) \| \nonumber  \\
=& \| (diag(\D  + \gamma \fOmega))^{-1} diag(\D) \| \nonumber  \\
=& \| (\D  + \gamma \fOmega)^{-1} \D \|. \label{eq:uppperDLomega}
\end{align}
The first, second and third equalities hold due to \eqref{eq:lem-lower-1},
\eqref{eq:lem-lower-2} and \eqref{eq:lem-lower-3}, respectively.
For the fourth equality is derived from  $diag(\L) = \O$.
The last equality holds, since $\D$ and $\fOmega$ are diagonal matrices.

We define
\begin{align*}
\xi =& \|(\D - \L + \gamma \fOmega)^{-1} \U\|,
\eta = \| (\D - \L + \gamma \fOmega)^{-1}(\gamma \fOmega-\D + \L)\|,  \\
\text{and} \ \mu =& \| (\D - \L + \gamma \fOmega)^{-1} \L \|
\end{align*}
and evaluate them.
First, we use \eqref{eq:uppperDLomega} to know an upper bound on $\xi$:
\begin{align*}
\begin{array}{rcl}
\xi &=&  \|(\D - \L + \gamma \fOmega)^{-1} \U\| \\
 &=&  \|((\D - \L + \gamma \fOmega)^{-1} (\D - \L)) ((\D-\L)^{-1} \U) \| \\
 &\le&  \|(\D - \L + \gamma \fOmega)^{-1} (\D - \L)\| \| (\D - \L)^{-1} \U \| \\
 &=& \| (\D  + \gamma \fOmega)^{-1} \D \| \tau.
\end{array}
\end{align*}
Similarly, for $\eta$, we derive
\begin{align*}
\begin{array}{rcl}
\eta&=&\|(\D - \L + \gamma \fOmega)^{-1}(\gamma \fOmega - \D + \L) \| \\
&=&  \|diag((\D - \L + \gamma \fOmega)^{-1}) diag(\gamma \fOmega - \D + \L) \| \\
&=&  \|diag((\D + \gamma \fOmega)^{-1}) diag(\gamma \fOmega - \D) \| \\
&=&  \|diag((\D + \gamma \fOmega)^{-1}(\gamma \fOmega - \D)) \|
=  \|(\D + \gamma \fOmega)^{-1}(\gamma \fOmega - \D) \|.
\end{array}
\end{align*}
Finally, we can show $\mu = 0$ as follows:
\begin{align*}
\mu =& \| (\D - \L + \gamma \fOmega)^{-1} \L \|
=  \| diag((\D - \L + \gamma \fOmega)^{-1}) diag(\L) \| \\
= & \| diag((\D - \L + \gamma \fOmega)^{-1}) \| \| diag(\L) \|
= \| diag((\D - \L + \gamma \fOmega)^{-1}) \| 0 = 0.
\end{align*}
Hence, from \eqref{eq:err-rel}, it holds
that
\begin{align}
& \| \x^{k+1}-\x^* \|  \nonumber \\
= & (2 \xi + \eta) \|\x^k-\x^* \|
+ \mu (\|\x^k - \x^*\| + \|\x^{k+1} - \x^*\|) \nonumber \\
\le & ( 2  \| (\D + \gamma \fOmega)^{-1} \D \| \tau  + \|(\D + \gamma \fOmega)^{-1}(\gamma \fOmega - \D) \| ) \|\x^k-\x^* \| \nonumber \\
= & \delta \|\x^k-\x^* \|. \label{eq:delta-sequence}
\end{align}
By the assumption $\delta < 1$ and the fixed-point theorem,
we obtain the convergence of $\{\x^k\}_{k = 0}^{\infty}$
and we know that $\x^*$ is the unique point to which
the sequence $\{\x^k\}_{k = 0}^{\infty}$ converges.
This completes the proof.
\hfill \qed
\end{proof}

Let us now focus the case $\fOmega = \alpha \D$ with the parameter $\alpha > 0$.
From Theorem~\ref{thm:amgs-conv}, we can derive the following corollary.

\begin{corollary}\upshape
	\label{col:alphaD}
Suppose $\fOmega = \alpha\D$.
If $\tau <1$ and $\alpha > \frac{\tau}{\gamma}$, the iteration sequence $\{\flambda ^k\}_{k=0}^\infty\subset\mathbb R^m$ generated by Method~\ref{mtd:proposedamgs} with
an arbitrary nonnegative initial vector $\flambda^0\in\mathbb R^m$
converges to the unique solution $\flambda^*\in\mathbb R^m$ of
$LCP(\q, \A)$.
\end{corollary}
\begin{proof}
	We put $\fOmega = \alpha\D$ to $\delta$ of
	Theorem~\ref{thm:amgs-conv}.
	Then, it holds that
\begin{align*}
\delta = & 2  \| (\D + \gamma \fOmega)^{-1} \D \| \tau + \|(\D + \gamma \fOmega)^{-1}(\gamma \fOmega - \D) \| \\
 = & 2  \| (\D + \alpha \gamma \D)^{-1} \D \| \tau  + \|(\D + \alpha \gamma \D)^{-1}(\alpha \gamma \D - \D) \| \\
 = & \frac{2\tau}{1+ \alpha \gamma} + \frac{|\alpha \gamma -1|}{1 + \alpha \gamma}.
\end{align*}
If $\alpha \gamma \ge 1$, it holds
$\delta = \frac{2 \tau - 1  + \alpha \gamma}{1+ \alpha \gamma}
<  1$, since $\tau < 1$.
On the other hand, if $\tau < \alpha \gamma < 1$, then
$\delta = \frac{2 \tau + 1  - \alpha \gamma}{1+ \alpha \gamma}
<  1$.
Therefore, we obtain $\delta < 1$. Thus, we can employ Theorem~\ref{thm:amgs-conv}
to derive this corollary.\hfill \qed
	\end{proof}

\begin{remark}\upshape
	\label{rem:bestalpha}
    Suppose $0 < \tau < 1$.
From \eqref{eq:delta-sequence}, we can evaluate
$\|\x^k - \x^*\| \le \delta^k \|\x^0 - \x^*\|$.
	In Corollary~\ref{col:alphaD},
	$\delta = \frac{2\tau}{1+ \alpha \gamma} + \frac{|\alpha \gamma -1|}{1 + \alpha \gamma}$ attains its minimum $\delta = \tau$ at
	$\alpha \gamma = 1$, since $0 < \tau < 1$ and a function $\frac{2\tau}{1+x} + \frac{|x-1|}{1+x}$ is monotonically decreasing when $0 < x < 1$ and monotonically increasing when $x > 1$.
	Therefore, when $\tau < 1$, a favorable choice on $\alpha$
	for  the case $\fOmega =  \alpha \D$ is
	$\alpha  = \frac{1}{\gamma}$.

	\end{remark}
\begin{remark}\upshape
    \label{rem:bestomega}
    If we choose $\fOmega = \bar{\omega} \E_m$ with $\bar{\omega} > 0$ as discussed in Zheng and Yin~\cite{zheng2013accelerated},
    $\delta$ in Theorem~\ref{thm:amgs-conv} can be evaluated as follows:
    \begin{align*}
    \delta =&  \max_{i=1,\ldots,m} \frac{2 D_{ii} \tau }{D_{ii} + \gamma \bar{\omega}}
    +
    \max_{i=1,\ldots,m} \frac{ |D_{ii} - \gamma \bar{\omega}|  }{D_{ii} + \gamma \bar{\omega}} \\
    \ge&  \max_{i=1,\ldots,m} \left\{ \frac{2 D_{ii} \tau }{D_{ii} + \gamma \bar{\omega}}
    +
    \frac{ |D_{ii} - \gamma \bar{\omega}|  }{D_{ii} + \gamma \bar{\omega}}\right\} \ge \tau.
    \end{align*}
    Therefore, a lower bound of $\delta$ is $\tau$, and this lower bound is attained
    only if $\gamma \bar{\omega} = D_{ii}$ for all $i=1,\ldots,m$.
    Thus, if $\D$ is not a multiple of the identity matrix, $\delta$ of the case
     $\fOmega = \frac{1}{\gamma}\D$ is smaller than that of $\fOmega = \bar{\omega} \E_m$.
     Hence, we can expect a faster convergence in  $\fOmega = \frac{1}{\gamma}\D$.
    \end{remark}

\section{Numerical experiments}
\label{sec:numexp}
In this section, we show numerical results of several 2-dimensional examples to verify the numerical performance of the proposed method.
All tests were performed on an Windows 8 computer with Intel Core i7-5500U (2.4 GHz CPU) and 8 GB memory space.

In the numerical experiments, we utilized the warm-start strategy \cite{erleben2004stable},
that is, the final solution obtained in a simulation step will be used as the first guess of the solution in the next simulation step.
More precisely, let $\left\{\flambda^{(t),k}\right\}_{k=0}^\infty$ denote
the sequence generated for solving $LCP(\q^{(t)}, \A^{(\t)})$ at a simulation step $t$.
We set the number of iterations in a single simulation step to 10, that is,
$\flambda^{(t),10}$ is used as the first guess $\flambda^{(t+1),0}$.
The initial point of the entire simulation is set as the zero vector $(\flambda^{(0),0} = \0)$.
To evaluate the accuracy of $\flambda^{(t),k}$,
we use a residual function in \cite{bai2010modulus}:
\begin{align}
\label{eq:residual-func}
RES(\flambda^{(t),k})=\left\|\min\{\flambda^{(t),k},\A\flambda^{(t),k}+\q\}\right\|.
\end{align}
For the AMGS methods, we set $\gamma = 2$,
and $\bar{\omega}$ is chosen as $\frac{\sum_{i=1}^m D_{ii}}{m \gamma}$
(the average of $D_{11}, \ldots, D_{mm}$ divided by $\gamma$)
based on preliminary experiments.

For the numerical experiments, we use examples ``Pool 1'', ``Pool 2'' and ``Stacking''; in the first two examples, rigid circles are stuffed into a small space, while, in the last case, rigid circles are vertically stacked.

\vspace{1em}
\noindent{\textbf{Pool 1}}

Figure~\ref{fig:pool1} displays an example with 221 circles in an area of 6 meters wide.
All the circles have the same mass (2 kilograms) and the same radius (21 centimeters). The coefficients of friction are set to 0.1, and the coefficients of restitution are set to 0.2.
When the coefficient of restitution is 1, collisions are perfectly elastic (\textit{i.e.}, no energy loss), and when the coefficient of restitution is 0, collisions are perfectly inelastic.
The gravitational acceleration is set to 9.80665 $\mathrm{m/s^2}$ in the downward direction in the figure.

In Pool~1, the size $n$ in the matrices $\J\in\Real^{m\times n}$
and $\M\in\Real^{n\times n}$ is $672$,
while the size $m$ depend on the simulation step $t$,
since the number of contact points between the circles changes as the simulation proceeds. During the entire simulation, the average of $m$ is $516$.

\begin{figure}[tbp]
	\centering
	\iffigure
	\includegraphics[width=0.5\textwidth]{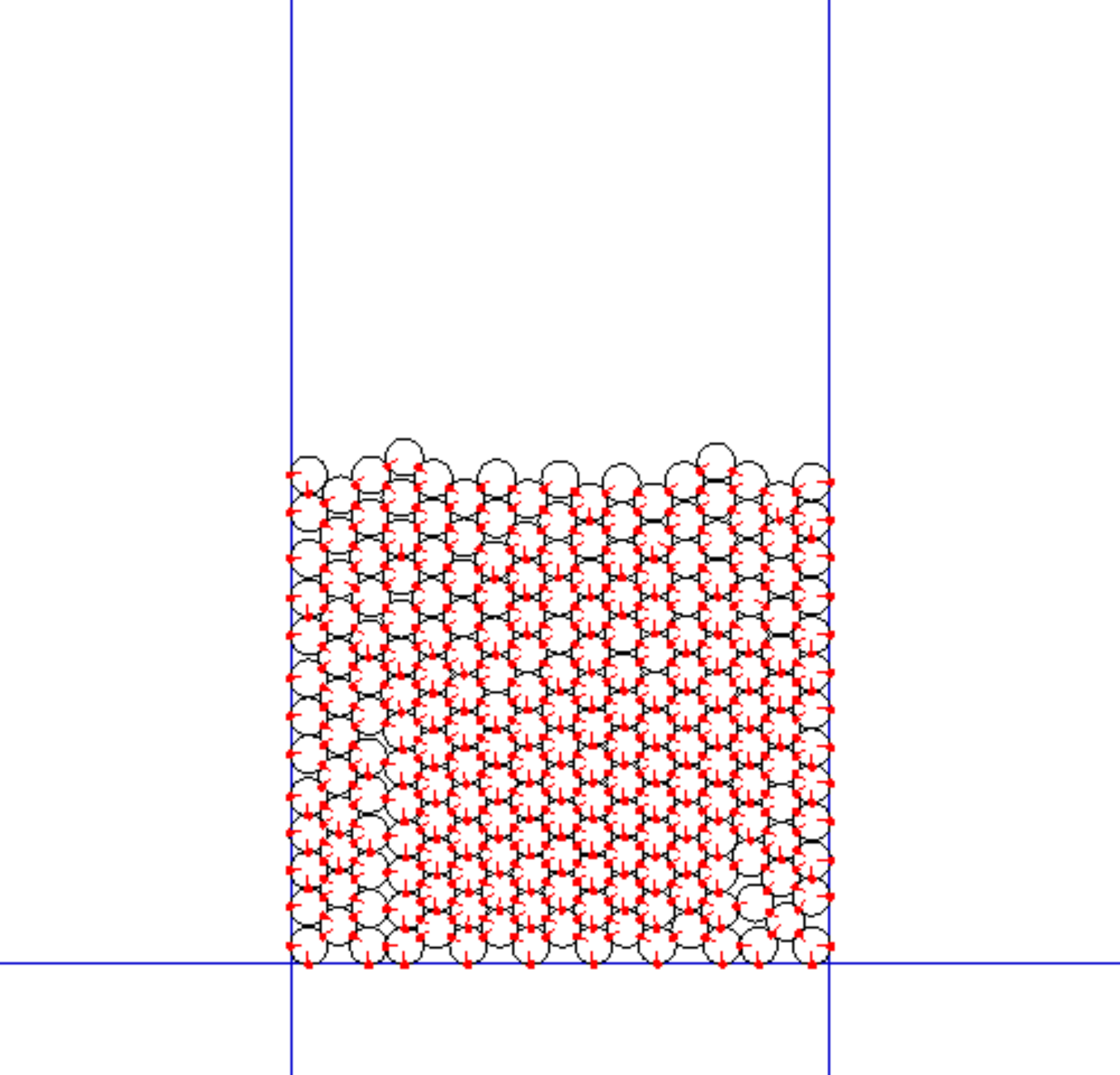}
	\fi
	\caption{A simulation of circles in Pool 1 and Pool 2.
		Small red points represent the contact points between the circles, and short red line segments the normal vectors at the contact points.}
	\label{fig:pool1}
\end{figure}

\vspace{1em}
\noindent{\textbf{Pool 2 }}

The most part of this example is same as Pool 1, but the masses of the circles increases linearly from 1.0 kilograms to 3.0 kilograms in accordance with their initial heights from the ground; circles in higher positions have larger masses than those in lower positions, and the mass of the heaviest circles (the top circles) are three times of that of the lightest circles (the bottom circles).
This is expected that the convergence will be slower, since it is hard for the lower (lighter) circles to support higher (heavier) circles.
The sizes $n$ and $m$ of the matrices are the same as Pool 1.

\vspace{1em}
\noindent{\textbf{Stacking}}

Figure~\ref{fig:stacking1} displays a simple example with 30 vertically stacked circles of 18-centimeter radius. All circles have the same masses (1.0 kilograms), and the coefficients of restitution are set to 0.2. The coefficients of friction are set to 0.1, but no frictional forces are produced because of the arrangement of the circles.
The sizes of $n$ and $m$ (average) are $48$ and $14$, respectively.

\begin{figure}[tbp]
	\centering
	\iffigure
	\includegraphics[width=0.5\textwidth]{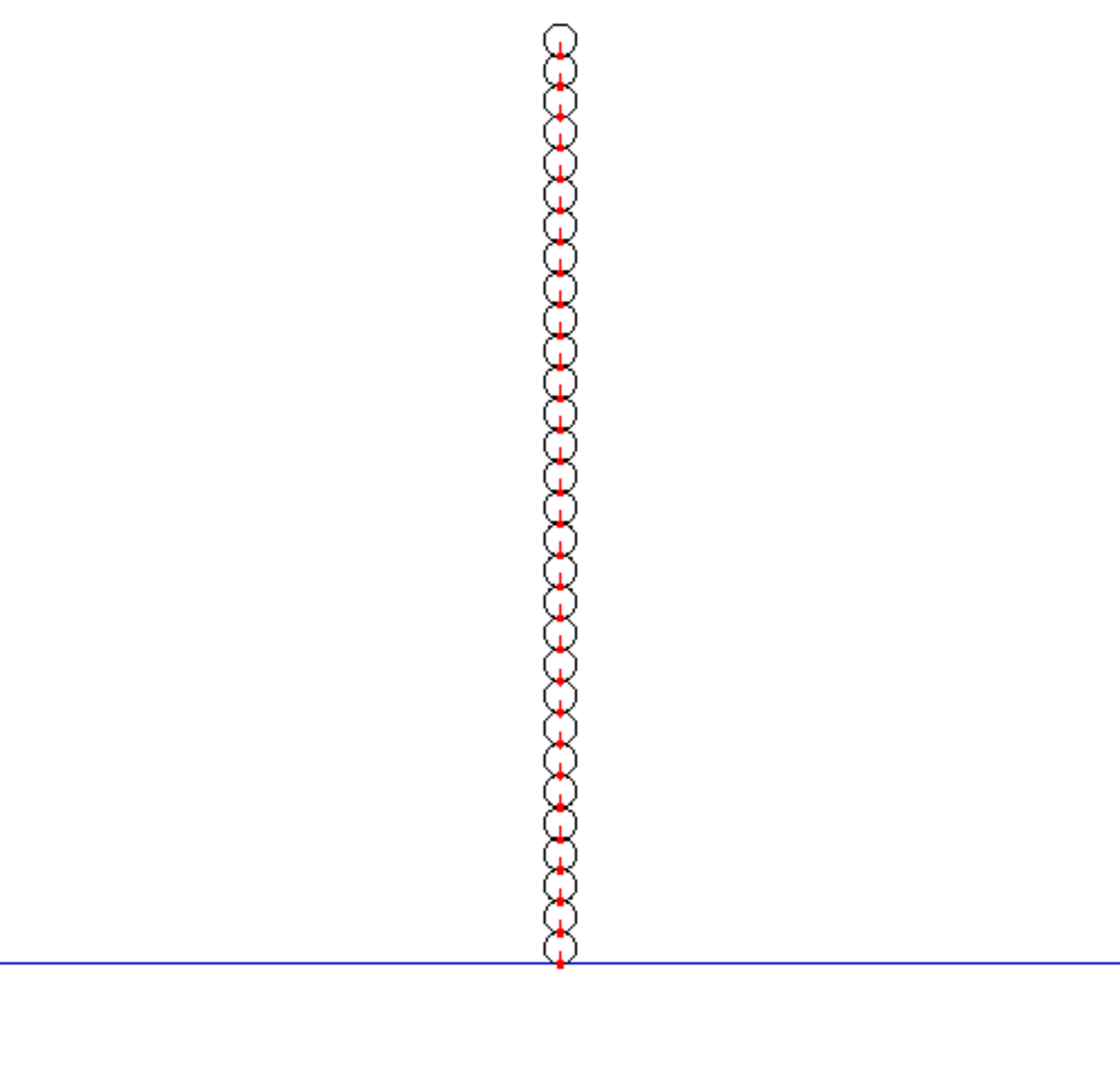}
	\fi
	\vspace{-0.4cm}
	\caption{A simulation of vertically stacked circles. Small red points represent the contact points of the circles, and short red line segments  the normal vectors at the contact points.}
	\label{fig:stacking1}
\end{figure}

\subsection{Convergence in each simulation step}
\label{sec:conv-each-step}
In each simulation step $t$, we execute only 10 iterations and we move to the next simulation step $t+1$ with the first guess $\flambda^{(t+1),0} = \flambda^{(t),10}$.
In this subsection, we execute more iterations to compare the convergence of the PGS method
and the proposed AMGS methods in each simulation step.
The average values of $\bar{\omega}$ in the numerical experiments are 1.928, 1.086, and 3.872 in Pool 1, Pool 2 and Stacking, respectively.

For Pool~1, Figure~\ref{fig:res1} plots $RES(\flambda^{(100),k})$ for $k=1,\ldots, 200$
of the PGS method, Method~\ref{mtd:proposedamgs} with $\fOmega = \bar{\omega} \E_m$
 (shortly, Method~\ref{mtd:proposedamgs}[$\bar{\omega} \E_m$])
and Method~\ref{mtd:proposedamgs} with $\fOmega = \frac{1}{\gamma} \D$ (shortly, Method~\ref{mtd:proposedamgs}[$\frac{1}{\gamma} \D$]).
The horizontal axis is the iteration number $k$ of $RES(\flambda^{(100),k})$.
In a similar way, Figures~\ref{fig:res2} and \ref{fig:res3} show
$RES(\flambda^{(100),k})$ of Pool~2 and Stacking, respectively.
We chose the 100th simulation step,
the early steps contained a lot of noise
and the warm-start strategy did not work effectively there.

\begin{figure}[tbp]
	\centering
	\iffigure
	\includegraphics[width=0.75\textwidth]{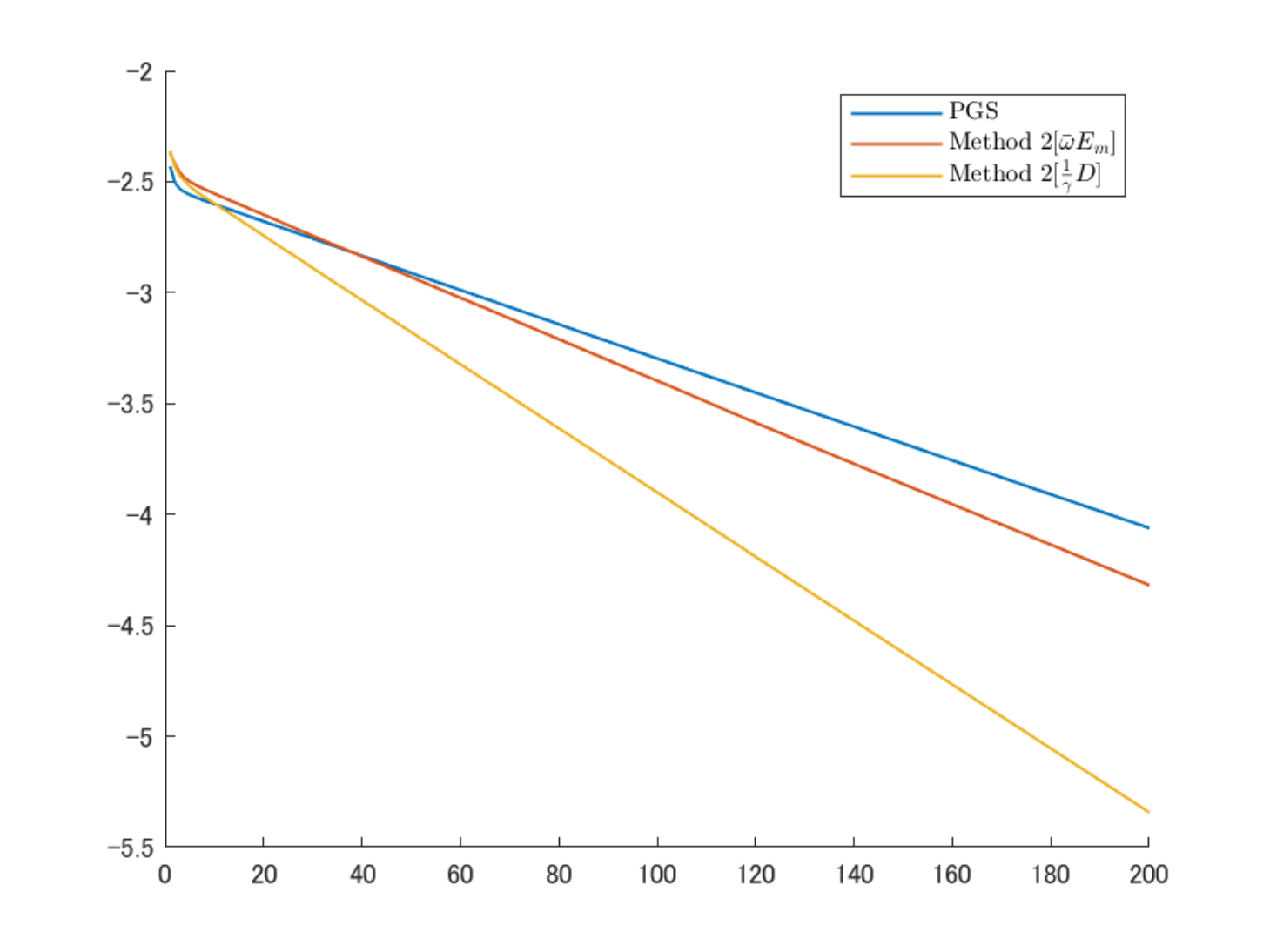}
	\fi
	\vspace{-0.4cm}
	\caption{Residuals of the sequence generated in the 100th simulation step for Pool 1.}
	\label{fig:res1}
\end{figure}
\begin{figure}[tbp]
	\centering
	\iffigure
	\includegraphics[width=0.75\textwidth]{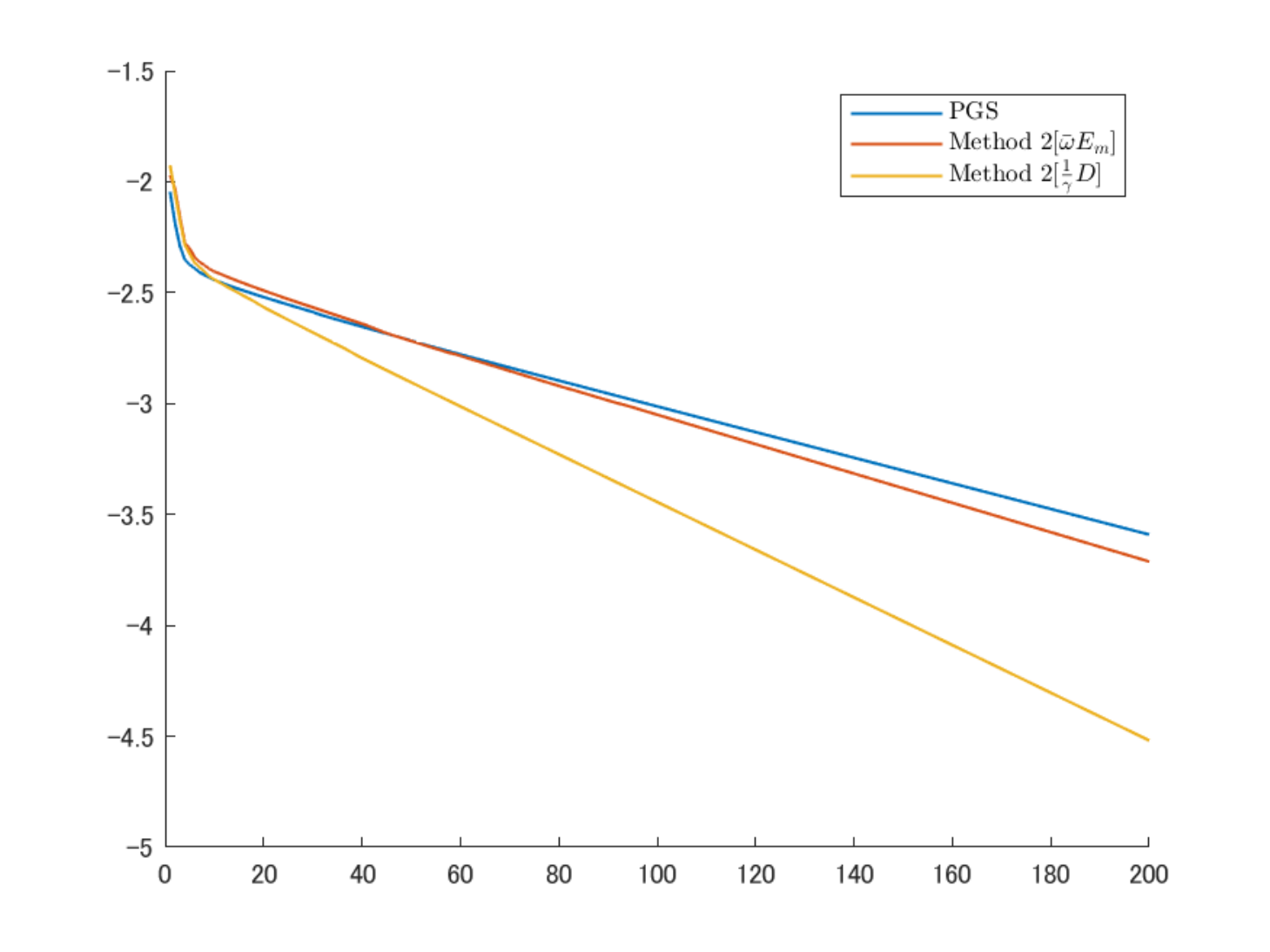}
	\fi
	\vspace{-0.4cm}
	\caption{Residuals of the sequence generated in the 100th simulation step for Pool 2.}
	\label{fig:res2}
\end{figure}
\begin{figure}[tbp]
	\centering
	\iffigure
	\includegraphics[width=0.75\textwidth]{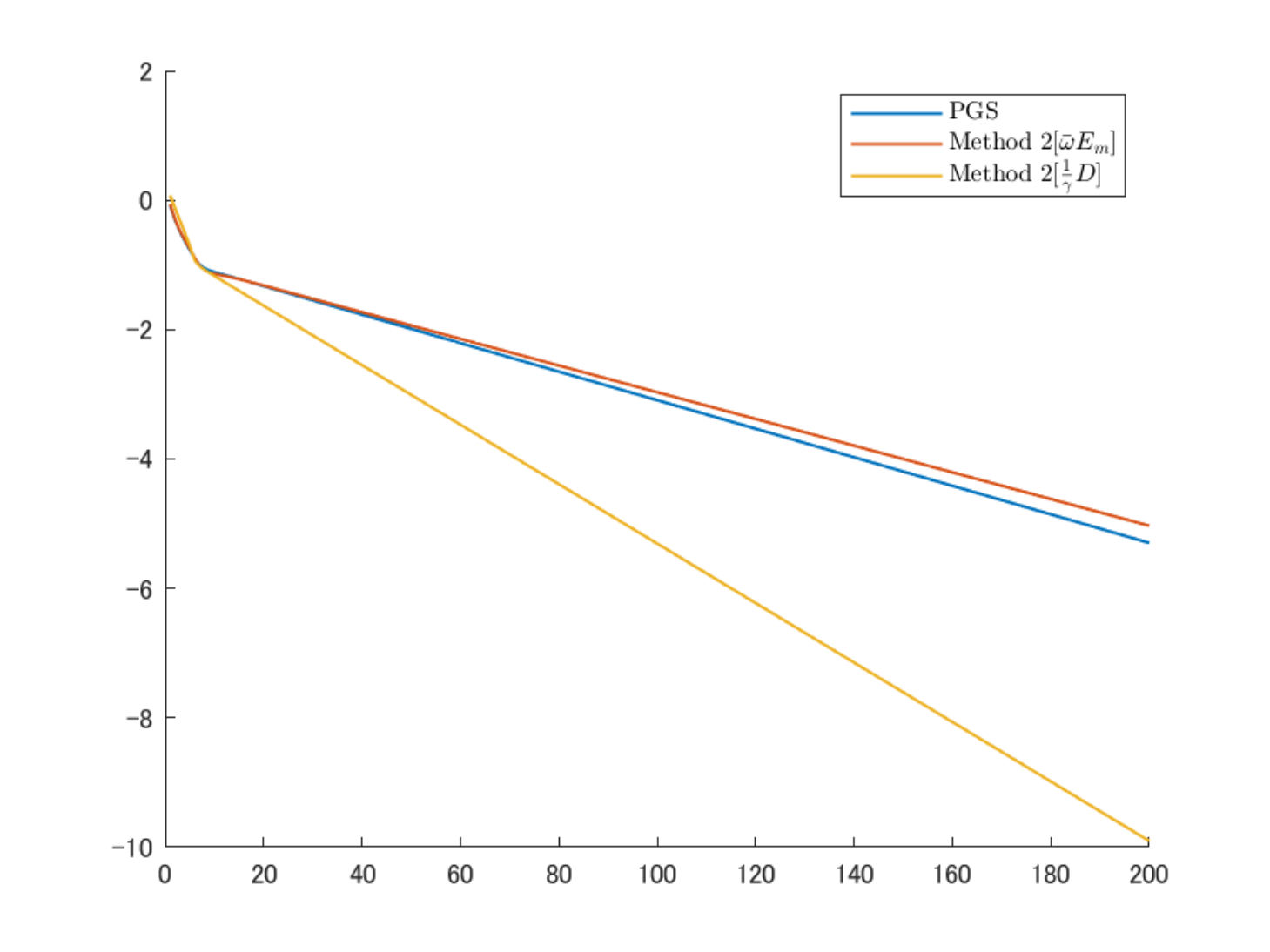}
	\fi
	\vspace{-0.4cm}
	\caption{Residuals of the sequence generated in the 100th simulation step for Stacking.}
	\label{fig:res3}
\end{figure}

From Figures~\ref{fig:res1}, \ref{fig:res2}, and \ref{fig:res3}, we observe that Method~\ref{mtd:proposedamgs} attains better convergence than the PGS method in most cases.

Table~\ref{table:comptime} reports the iteration number $k$ and the computation time in seconds of each method
to reach $RES(\flambda^{(100),k}) < 10^{-4}$.
Since the standard AMGS method cannot directly handle contacts with frictions, coefficients of friction are set to 0 only for this experiment. Also, the computation time for the 30 circles in Stacking (Figure~\ref{fig:stacking1}) was too short to measure (shorter than $0.001$ seconds), therefore we used 150 circles instead of 30 circles.
Method~\ref{mtd:amgs} computes the coefficient matrix $\A$ explicitly, and we observe that Method~\ref{mtd:amgs} is the slowest in Table~\ref{table:comptime}.
This computation time is insufficient for real-time simulations.

In the comparison between Method~\ref{mtd:proposedamgs}[$\bar{\omega} \E_m$] and Method~\ref{mtd:proposedamgs}[$\frac{1}{\gamma} \D$],
we can see that Method~\ref{mtd:proposedamgs}[$\frac{1}{\gamma} \D$] achieves
better convergences than Method~\ref{mtd:proposedamgs}[$\bar{\omega} \E_m$],
and this result is consistent with Remarks~\ref{rem:bestalpha} and \ref{rem:bestomega}. Furthermore, Method~\ref{mtd:proposedamgs}[$\frac{1}{\gamma} \D$] achieves the smallest number of iterations and computation times in all three cases.

\begin{table}[tbp]
	\caption{The iteration number and the computation time to reach $10^{-4}$.}
	\label{table:comptime}
	\centering
	\begin{tabular}{lrr}
		\hline
\multicolumn{3}{c}{Pool 1} \\
\hline
& iteration & time (seconds) \\
\hline
PGS method	& 1452 & 0.042 \\
Method~\ref{mtd:amgs}[$\frac{1}{\gamma} \D$] & 766 & 1.251 \\
Method~\ref{mtd:proposedamgs}[$\bar{\omega} \E_m$]
& 1522 & 0.055 \\
Method~\ref{mtd:proposedamgs}[$\frac{1}{\gamma} \D$]
& 766 & 0.023  \\
\hline
		\hline
\multicolumn{3}{c}{Pool 2} \\
\hline
& iteration & time (seconds) \\
\hline
PGS method	& 1359 & 0.042 \\
Method~\ref{mtd:amgs}[$\frac{1}{\gamma} \D$] & 744 & 1.428 \\
Method~\ref{mtd:proposedamgs}[$\bar{\omega} \E_m$]
& 1393 & 0.048 \\
Method~\ref{mtd:proposedamgs}[$\frac{1}{\gamma} \D$]
& 744 & 0.028  \\
\hline
		\hline
\multicolumn{3}{c}{Stacking} \\
\hline
& iteration & time (seconds) \\
\hline
PGS method	& 16186 & 0.162 \\
Method~\ref{mtd:amgs}[$\frac{1}{\gamma} \D$] & 7754 & 0.293 \\
Method~\ref{mtd:proposedamgs}[$\bar{\omega} \E_m$]
& 16431 & 0.183 \\
Method~\ref{mtd:proposedamgs}[$\frac{1}{\gamma} \D$]
& 7754 & 0.066  \\
\hline
\end{tabular}
\end{table}

\subsection{Convergence in entire simulation}
\label{sec:conv-entire}

In this subsection, we report the computation errors $RES(\flambda^{(t), 10})$
along with the progress of simulation steps $t \ge 60$.
We removed the first 60 steps from the figure,
since the early steps contained a lot of noise.
In the previous subsection, we observed
that Method~\ref{mtd:proposedamgs}[$\frac{1}{\gamma} \D$]
is superior to
Method~\ref{mtd:amgs} and
Method~\ref{mtd:proposedamgs}[$\bar{\omega} \E_m$]
in each simulation step. Thus, we use only
Method~\ref{mtd:proposedamgs}[$\alpha \D$]
in this subsection, changing the value of $\alpha$.

In Figure~\ref{fig:pool1res}, the horizontal axis is the simulation step $t$,
and the vertical axis is the residual $RES(\flambda^{(t), 10})$.
From Figure~\ref{fig:pool1res}, we observe that Method~\ref{mtd:proposedamgs} converges faster than the PGS method for $t \ge 300$.
Among different values of $\alpha$, $\alpha=0.2$ shows the fastest convergence in the figure. In Remark~\ref{rem:bestalpha}, we analyzed that $\alpha = \frac{1}{\gamma}$ (here, we used $\gamma = 2$,  so $\frac{1}{\gamma} = 0.5$)
leads to the smallest $\delta$ such that $||\x^k - \x^*||\le \delta^k ||\x^0 - \x^*||$. However, this is only an upper bound of $||\x^k - \x^*||$,
and Figure~\ref{fig:pool1res} implies that a faster convergence is possible by choosing a smaller $\alpha$.

\begin{figure}[tb]
	\centering
\iffigure
	\includegraphics[width=0.75\textwidth]{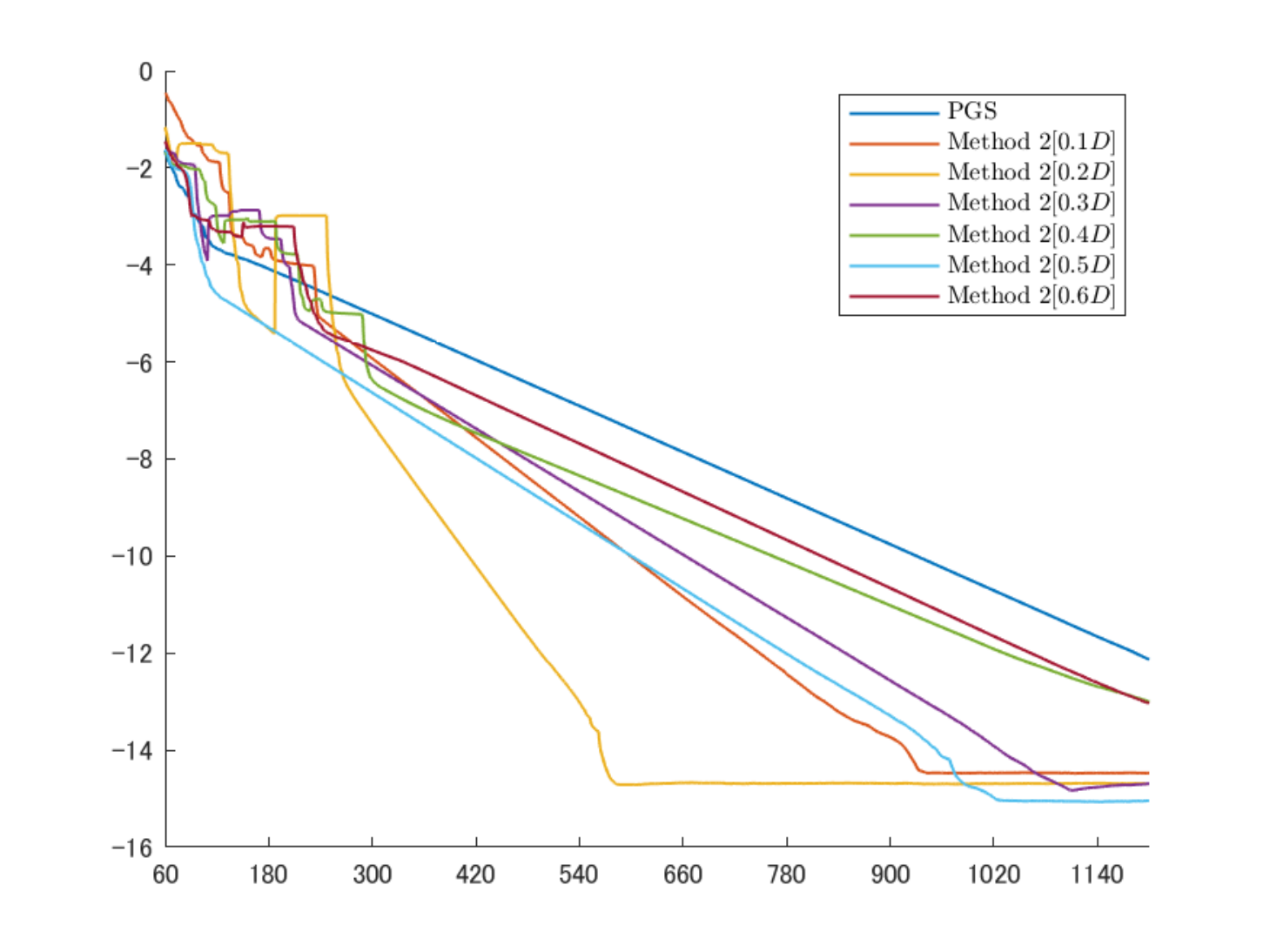}
\fi
	\caption{The residuals in the entire simulation of Pool 1.}
	\label{fig:pool1res}
\end{figure}

The result of Pool 2 illustrated in Figure~\ref{fig:pool2res} indicates that there are no clear differences of the convergence speed between the PGS method and the AMGS method with various values of $\alpha$, but when $\alpha=0.1$, the simulation is unstable during about the first 100 simulation steps.

\begin{figure}[tbp]
	\centering
\iffigure
	\includegraphics[width=0.75\textwidth]{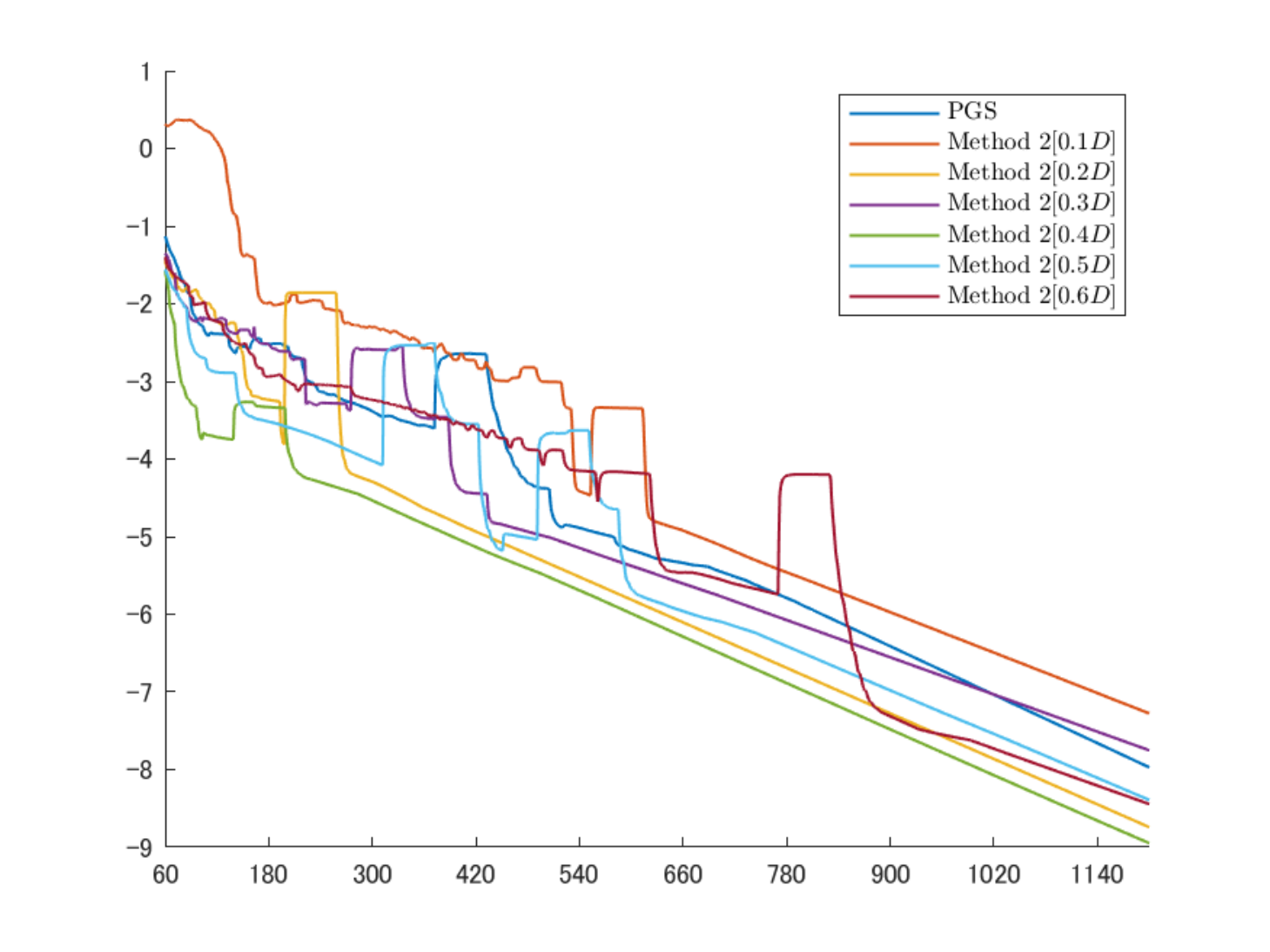}
\fi
	\caption{The residuals in the entire simulation of Pool 2.}
	\label{fig:pool2res}
\end{figure}

From Figure~\ref{fig:stacking1res} for Stacking, in a similar way to Pool 1, we can again observe that the AMGS method with the smaller $\alpha$ gives the faster convergence, and the AMGS method with $\alpha=0.6$ still converges faster than the PGS method.

Finally, Table~\ref{table:comptimeentire} shows the entire computation time for 1,000 simulation steps, with 200 iterations for each step, that is,
    we computed $\flambda^{(1000),200}$.
The entire computation time is still shorter in the proposed method than in the PGS method.
As mentioned in Table~\ref{table:comptime}, the convergence rate is better in the proposed method (Method~\ref{mtd:proposedamgs}[$\alpha \D$]) than in the PGS method, in other words, the proposed method simulates the circles more accurately than the PGS method. Therefore, the proposed method has the advantages for interactive rigid-body simulations.
\begin{table}[tbp]
    \caption{The computation time for 1,000 simulation steps.}
    \label{table:comptimeentire}
    \centering
    \begin{tabular}{lr}
        \hline
\multicolumn{2}{c}{Pool 1} \\
\hline
&  time (seconds) \\
\hline
PGS method	&  2.639 \\
Method~\ref{mtd:proposedamgs}[$0.1 \D$] & 2.083 \\
Method~\ref{mtd:proposedamgs}[$0.2 \D$] & 1.974 \\
Method~\ref{mtd:proposedamgs}[$0.3 \D$] & 2.051 \\
Method~\ref{mtd:proposedamgs}[$0.4 \D$] & 1.996 \\
Method~\ref{mtd:proposedamgs}[$0.5 \D$] & 1.968 \\
Method~\ref{mtd:proposedamgs}[$0.6 \D$] & 2.022 \\
\hline
        \hline
\multicolumn{2}{c}{Pool 2} \\
\hline
&  time (seconds) \\
\hline
PGS method	&  2.093 \\
Method~\ref{mtd:proposedamgs}[$0.1 \D$] & 1.921 \\
Method~\ref{mtd:proposedamgs}[$0.2 \D$] & 1.907 \\
Method~\ref{mtd:proposedamgs}[$0.3 \D$] & 1.849 \\
Method~\ref{mtd:proposedamgs}[$0.4 \D$] & 1.904 \\
Method~\ref{mtd:proposedamgs}[$0.5 \D$] & 1.915 \\
Method~\ref{mtd:proposedamgs}[$0.6 \D$] & 1.915 \\
\hline
        \hline
\multicolumn{2}{c}{Stacking} \\
\hline
&  time (seconds) \\
\hline
PGS method	&  0.095 \\
Method~\ref{mtd:proposedamgs}[$0.1 \D$] & 0.099 \\
Method~\ref{mtd:proposedamgs}[$0.2 \D$] & 0.081 \\
Method~\ref{mtd:proposedamgs}[$0.3 \D$] & 0.085 \\
Method~\ref{mtd:proposedamgs}[$0.4 \D$] & 0.085 \\
Method~\ref{mtd:proposedamgs}[$0.5 \D$] & 0.089 \\
Method~\ref{mtd:proposedamgs}[$0.6 \D$] & 0.082 \\
\hline
    \end{tabular}
\end{table}


\begin{figure}[tbp]
	\centering
\iffigure
	\includegraphics[width=0.75\textwidth]{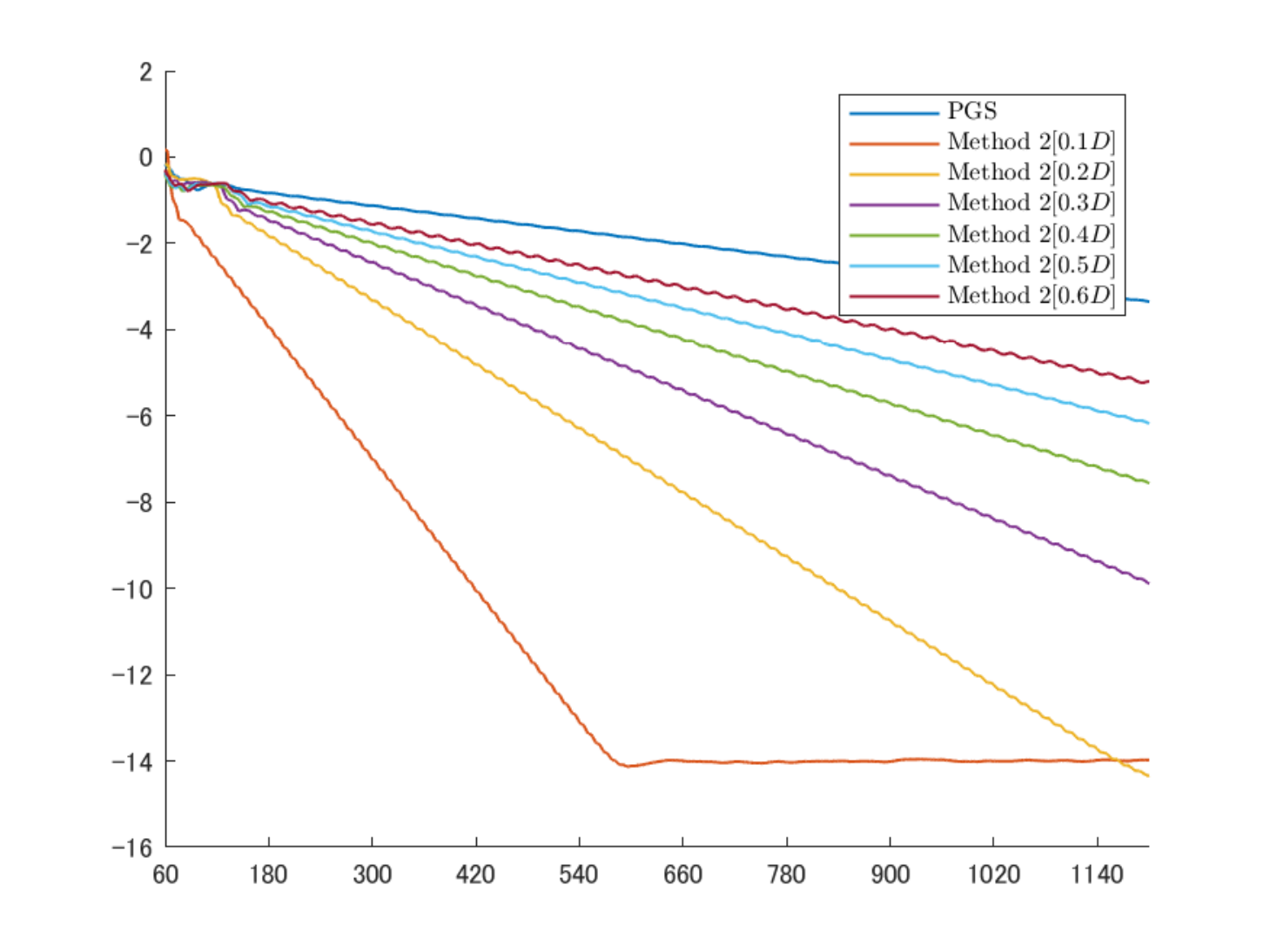}
\fi
	\caption{The residuals in the entire simulation of Stacking.}
	\label{fig:stacking1res}
\end{figure}

\section{Conclusion}\label{sec:conclusion}
We presented a numerical method based on the AMGS method for interactive rigid-body simulations exploiting the sparse structure of the data matrices.
We established the convergence theorem of the AMGS method for the case the matrix $\A$ is positive definite and $\fOmega=\alpha \D$ with $\alpha>0$.
This case was examined in  the numerical experiments, and we
observed that the proposed method attained the better accuracy than the PGS method and the computation time of the proposed method was shorter than that of
a simple application of the AMGS method.

In practical cases, however, determining a proper value of $\alpha$ is not simple.
As discussed in Remark~\ref{rem:bestalpha}, we should choose
$\alpha = \frac{1}{\gamma}$ to minimize $\delta$ in Theorem~\ref{thm:amgs-conv}.
However, this $\delta$ is just a theoretical upper bound.
Actually, the numerical results showed that a smaller value of $\alpha$ gave
a better convergence.
 An approach that adaptively determines the value of $\alpha$ may resolve this problem, and we leave a discussion on such an approach as a future task of this paper.
Further numerical experiments in 3-dimensional spaces that take frictions into consideration will be another topic of our future studies.


%
%

%
%

\end{document}